\documentclass{amsart}

\usepackage[all]{xy}
\usepackage{amsmath}
\usepackage{hyperref}
\usepackage{amsfonts,graphics,amsthm,amsfonts,amscd,latexsym}
\usepackage{epsfig}
\usepackage{flafter}
\usepackage{mathtools}
\usepackage{comment}

\hypersetup{
    colorlinks=true,    
    linkcolor=blue,          
    citecolor=blue,      
    filecolor=blue,      
    urlcolor=blue           
}
\usepackage{tikz}
\usetikzlibrary{graphs,positioning,arrows,shapes.misc,decorations.pathmorphing}
\usepackage{tikz-cd}

\tikzstyle{vertex}=[
    draw,
    circle,
    fill=black,
    inner sep=1pt,
    minimum width=5pt,
]
\usepackage[position=top]{subfig}
\usepackage[alphabetic,backrefs]{amsrefs}
\usepackage{amssymb}
\usepackage{color}

\calclayout

\usetikzlibrary{decorations.pathmorphing}
\tikzstyle{printersafe}=[decoration={snake,amplitude=0pt}]

\DeclareFontFamily{U}{wncy}{}
\DeclareFontShape{U}{wncy}{m}{n}{<->wncyr10}{}
\DeclareSymbolFont{mcy}{U}{wncy}{m}{n}
\DeclareMathSymbol{\Sh}{\mathord}{mcy}{"58} 

\newcommand{\Pic}{\operatorname{Pic}}

\newcommand{\codim}{\operatorname{codim}}

\newcommand{\mult}{\operatorname{mult}}

\newcommand{\vol}{\operatorname{vol}}

\renewcommand{\qq}{\mathbb{Q}}
\newcommand{\zz}{\mathbb{Z}}
\newcommand{\nn}{\mathbb{N}}

\def\O#1.{\mathcal {O}_{#1}}			
\def\pr #1.{\mathbb P^{#1}}				
\def\af #1.{\mathbb A^{#1}}			
\def\ses#1.#2.#3.{0\to #1\to #2\to #3 \to 0}	
\def\xrar#1.{\xrightarrow{#1}}			
\def\K#1.{K_{#1}}						
\def\bA#1.{\mathbf{A}_{#1}}			
\def\bM#1.{\mathbf{M}_{#1}}
\def\bN#1.{\mathbf{N}_{#1}}
\def\bL#1.{\mathbf{L}_{#1}}				
\def\bB#1.{\mathbf{B}_{#1}}				
\def\bK#1.{\mathbf{K}_{#1}}			
\def\subs#1.{_{#1}}					
\def\sups#1.{^{#1}}						

\DeclareMathOperator{\coeff}{coeff}	

\DeclareMathOperator{\Supp}{Supp}

\DeclareMathOperator{\lct}{lct}

\newcommand{\rar}{\rightarrow}	
\newcommand{\drar}{\dashrightarrow}	

\usepackage{tikz}
\usetikzlibrary{matrix,arrows,decorations.pathmorphing}

  \newtheorem{theorem}{Theorem}[section]
  \newtheorem{lemma}[theorem]{Lemma}
  \newtheorem{proposition}[theorem]{Proposition}
  \newtheorem{corollary}[theorem]{Corollary}

\theoremstyle{definition}

  \newtheorem{definition}[theorem]{Definition}
  \newtheorem{example}[theorem]{Example}

\newtheorem{remark}[theorem]{Remark}

\theoremstyle{remark}

\numberwithin{equation}{section}

\begin{document}

\title{On the boundedness of $n$-folds with $\kappa(X)=n-1$}

\author[S.~Filipazzi]{Stefano Filipazzi}
\address{
    EPFL SB MATH CAG,
	MA  C3 625  (B\^atiment MA),
	Station 8,
	CH-1015 Lausanne,
    Switzerland
}
\email{stefano.filipazzi@epfl.ch}

\renewcommand{\subjclassname}{%
	\textup{2020} Mathematics Subject Classification}

\maketitle

\begin{abstract}
In this note we study certain sufficient conditions for a set of minimal klt pairs $(X,\Delta)$ with $\kappa(X,\Delta)=\dim(X)-1$ to be bounded.
\end{abstract}

\setcounter{tocdepth}{1}

\tableofcontents

\section{Introduction}

Throughout this paper, the base field will be an algebraically closed field of characteristic zero.

One of the main goals of algebraic geometry is to classify varieties according to a few main features.
One natural object attached to any normal variety $X$ is its canonical sheaf $\omega_X$.
A fruitful perspective is to group varieties according to the behavior of the canonical sheaf.
Under this perspective, the Minimal Model Program suggests the existence of three main classes of varieties: Fano varieties, Calabi--Yau varieties, and varieties of general type.
Each of these classes corresponds to one of the following behaviors of the canonical sheaf: $\omega_X^\vee$ induces a birational polarization, $\omega_X$ is torsion, and $\omega_X$ induces a birational polarization, respectively.
Notice that, for the purpose of this work, Calabi--Yau will refer to the torsion property of the (log) canonical divisor, while there will be no assumption on the fundamental group, nor on the smoothness of the variety.
Then, these notions naturally extend to the setup of pairs.

The Minimal Model Program conjectures that every variety decomposes birationally as iterated fibrations, where the base of the tower of fibrations and the general fiber of each fibration belong to one of the three above families.
Thus, a natural point of view is to use the knowledge of the three fundamental classes of varieties to have a deeper understanding of other types of varieties.

In this work, we are concerned with varieties of intermediate Kodaira dimension.
That is, we consider varieties $X$ so that $\Gamma(X,\omega_X^{\otimes n}) \neq 0$ for some $n \in \mathbb N$, but that are not Calabi--Yau nor of general type.
This is equivalent to the following: the rate of growth of the sections of $\Gamma(X,\omega_X^{\otimes n})$ is asymptotically polynomial, with degree $d$ satisfying $1 \leq d \leq \dim(X)-1$.
In this scenario, the Minimal Model Program predicts that $X$ admits a minimal model $X'$ with the following property: a suitable positive power of $\omega_{X'}$ is basepoint-free and defines a morphism $f \colon X' \rar Y$, called the Iitaka fibration.
In this case, the fibers of $f$ are Calabi--Yau varieties and the base $Y$ is naturally endowed with the structure of a pair of general type.

One of the main topics in the classification of algebraic varieties is boundedness.
Loosely speaking, a set of varieties or, more generally, pairs is called bounded if it can be parametrized by a scheme of finite type.
In particular, addressing the boundedness of a class of varieties is the first step towards considering a moduli space.
Hacon, M\textsuperscript{c}Kernan, and Xu proved that pairs $(X,\Delta)$ of general type form a bounded family when one imposes restrictions on $\dim(X)$, the coefficients of $\Delta$, and the divisor $\K X. +\Delta$ is ample of fixed volume \cite{HMX18}.
Similarly, Birkar settled an optimal boundedness statement for varieties of Fano type: varieties of Fano type whose singularities are bounded form a bounded family \cites{Bir19,Bir16}.
On the other hand, in general, Calabi--Yau varieties do not form bounded families.
An example of this phenomenon is given by abelian $n$-folds and K3 surfaces: both classes are not bounded, but each of them decomposes as the countable union of bounded families, corresponding to polarizations of different degrees.
Therefore, it is natural to introduce some geometrical or numerical conditions when studying families of Calabi--Yau varieties.
In this direction, there are some results obtaining (weak versions of) boundedness of Calabi--Yau varieties.
The boundedness of Calabi--Yau varieties admitting elliptic fibrations is considered in \cites{Gro94,dCS17,Bir18,BDCS,FHS}, while some recent works consider fairly singular Calabi--Yau pairs \cites{Bir18,dCS18,HJ}.
Furthermore, works of Jiao consider the problem of the boundedness of varieties admitting a log Calabi--Yau fibration of higher relative dimension \cites{Jiao1,Jiao2}.

In this paper, we investigate the following question.
If the base $Y$ and the general fiber $F$ of the Iitaka fibration $f \colon X' \rar Y$ belong to a bounded family of varieties, can we infer any boundedness statement about $X'$?
If not, what are some natural additional conditions to achieve boundedness?
In this direction, recent work of Li analyzes the case when the Iitaka fibration has fibers of Fano type \cite{Li20}.
Similarly, the results in \cite{Bir18} apply to the study of fibrations of Fano type.

In this work, we are concerned with the case of varieties $X$ of Kodaira dimension $\dim (X) -1$.
More generally, we are interested in minimal klt pairs $(X,\Delta)$, where $\kappa(X,\Delta)=\dim(X)-1$.
In this case, the general fiber of the Iitaka fibration is either an elliptic curve or $\pr 1.$.
While the case when the general fiber is $\pr 1.$ follows from \cite{Bir18}, substantial work is needed to analyze the case of elliptic fibrations.
If the elliptic fibration admits a section, using techniques developed in \cite{dCS17}, one can induce a polarization that bounds the total space of the fibration.
This direction is successfully explored in \cite{FS19}.
On the other hand, an elliptic fibration does not necessarily admit a rational section.
Furthermore, if a set of pairs is bounded by a family $\mathcal{X} \rar T$, one can stratify $T$ so that the relative Iitaka fibration of $\mathcal{X} \rar T$ induces the Iitaka fibration fiberwise.
In particular, a multisection of the relative Iitaka fibration induces, up to a stratification, a multisection of the Iitaka fibration of the fibers.
Thus, if a set of elliptic $n$-folds with $\kappa(X)=n-1$ is bounded, the minimal degree of a multisection is bounded.
Similarly, the bases of the Iitaka fibration have to be bounded too.

Our main result shows that the above constraints, namely on the degree of the multisection and on the boundedness of the bases, are actually sufficient to achieve boundedness in codimension one.
We refer to \S~\ref{hpyerstandard} for the definition of the set of coefficients $\Phi(\mathcal{R})$.

\begin{theorem} \label{thm intro}
Fix positive integers $n$ and $d$, a positive real number $v$, and a finite set of rational numbers $\mathcal{R} \subset \qq \cap [0,1]$.
Let $\mathfrak{D}(n,v,\Phi(\mathcal{R}),d)$ be the set of minimal projective klt pairs $(X,\Delta)$ of dimension $n$ with $\kappa(X,\Delta)=n-1$, $\vol_{n-1}(X,\K X. + \Delta)=v$, $\coeff(\Delta) \subset \Phi(\mathcal{R})$, and whose Iitaka fibration admits a multisection of degree $d$.
Then, $\mathfrak{D}(n,v,\Phi(\mathcal{R}),d)$ is bounded in codimension one.
\end{theorem}

\begin{remark}
The notation $\vol_{n-1}(X,\K X. + \Delta)=v$ is defined in Definition \ref{def volume}.
It is the volume of the pair of general type naturally induced on the base of the Iitaka fibration.
\end{remark}

\begin{remark}
The varieties in the statement of Theorem \ref{thm intro} are bounded together with their Iitaka fibrations.
This is made precise by the more technical statements Theorem \ref{theorem fano case} and Theorem \ref{main thm}.
\end{remark}

\begin{remark}
If the general fiber of the Iitaka fibration is $\pr 1.$, the statement of Theorem \ref{thm intro} is stronger: boundedness (not just in codimension one) holds, and the requirement on the degree of the multisection is unnecessary (this is consistent with the fact that a $\pr 1.$-fibration always has a multisection of degree 2).
See Theorem~\ref{theorem fano case}.
\end{remark}

As mentioned above, the main case of Theorem \ref{thm intro} is when the Iitaka fibration $f \colon X \rar Y$ is an elliptic fibration.
In this case, we can consider an auxiliary fibration, the Jacobian fibration $j \colon J(X) \rar Y$, whose generic fiber is the Jacobian of the generic fiber of $f$.
In particular, $j$ admits a rational section.
Thus, one can argue as in \cites{dCS17,FS19}, and conclude that the set of Jacobian fibrations is bounded in codimension one.
To retrieve the original fibration, we make use of tools developed by Dolgachev and Gross \cites{DG94,Gro94}.
More precisely, we consider the geometric Tate--Shafarevich group of the Jacobian fibration.
This group parametrizes certain elliptic fibrations over $Y$ whose Jacobian fibration is $j$.
Gross showed that the torsion parts of this group behave well in family:
Roughly speaking, for every $d$, the $d$-torsion elements of this group form a finite set, which is constant in the \'etale topology under small deformations.
In particular, one could hope to retrieve the whole family of fibrations from the family of Jacobians.
On the other hand, the fibrations parametrized by this group have very restrictive geometric properties, which our fibrations only satisfy over some open set of the base.
The main technical statement of this work is to bound the complement of this subset, in order to show that it deforms along with the family of Jacobian fibrations.
This is the content of Proposition \ref{bound cycles}.

In a similar flavor as \cite{HMX18}, Theorem \ref{thm intro} is concerned with minimal models.
Thus, it is interesting to ask whether any klt pair $(X,\Delta)$ with $\kappa(X,\Delta)=n-1$ has a minimal model to which Theorem \ref{thm intro} applies.
The Minimal Model Program predicts that it should be the case.
More precisely, every klt pair of non-negative Kodaira dimension is expected to admit a good minimal model.
In general, this statement is known just in some cases, for instance in the case of varieties of general type.
In this work, we show the existence of good minimal models when the fibers of the Iitaka fibration have dimension up to 3.

\begin{theorem} \label{good mm}
Let $(X,\Delta)$ be a projective klt pair of dimension $n$ with rational coefficients with $\kappa(X,\Delta) \geq 0$.
Assume that $\kappa(X,\Delta) \geq n-3$.
Then, $(X,\Delta)$ has a good minimal model.
\end{theorem}

Theorem \ref{good mm} is a generalization of ideas of Lai \cite{Lai11}, who considered terminal varieties with no boundary.
The specific statement of Theorem \ref{good mm} has already been observed in the case of smooth varieties without boundary \cite{HS20}*{Theorem 2.1}.
Also, the strategy of the proof generalizes the approach in \cite{GW19}, where the case of elliptic fibrations is considered.
More precisely, we first consider a resolution of indeterminacies of the Iitaka fibration.
Then, by the canonical bundle formula, we can induce the structure of a pair of general type on the base of the Iitaka fibration.
On the one hand, by \cite{BCHM}, we can run an MMP on the base of the Iitaka fibration, which terminates with a good minimal model $Y^m$.
On the other hand, since the abundance conjecture is known in dimension up to 3, we can apply the results of \cite{HX13}, which guarantee that our variety $X$ has a relative good minimal model over the base.
Combining these two facts, the proof is completed.

\subsection*{Acknowledgements} 
The author would like to thank Gabriele Di Cerbo and Roberto Svaldi for many fruitful discussions on this work.
He would like to thank Stefan Patrikis for explaining the basics of Tate--Shafarevich groups.
He would like to thank Christopher Hacon and Burt Totaro for helpful comments and feedback, and Javier Carvajal-Rojas, Mark Gross and J\'anos Koll\'ar for answering his questions.
Finally, he would like to thank the anonymous referee for the thorough feedback and the many suggestions to improve this work.

\section{Preliminaries}

\subsection{Conventions} Throughout this paper, the base field will be an algebraically closed field of characteristic zero.

\subsection{Contractions} In this paper a {\em contraction} is a projective morphism of quasi-projective varieties $f  \colon  X \rar Z$ with $f_* \O X. = \O Z.$. Notice that, if $X$ is normal, then so is $Z$.
An \emph{elliptic fibration} is a contraction whose general fiber is a smooth elliptic curve.

\subsection{Hyperstandard sets}\label{hpyerstandard}

Let $\mathcal R$ be a subset of $[0,1]$.
Then, we define the \emph{set of hyperstandard multiplicities} associated to $\mathcal R$ as
\[
\Phi (\mathcal R) \coloneqq \bigg \lbrace \left. 1- \frac{r}{m}  \right| r \in \mathcal R, m \in \nn \bigg \rbrace.
\]
When $\mathcal R = \lbrace 0,1 \rbrace$, we call it the set of \emph{standard multiplicities}.
Usually, with no mention, we assume $0,1 \in \mathcal R$, so that $\Phi (\lbrace 0,1 \rbrace ) \subset \Phi (\mathcal R)$.
Furthermore, if $1-r \in \mathcal R$ for every $r \in R$, we have that $\mathcal R \subset \Phi (\mathcal R)$.
Now, assume that $\mathcal{R} \subset [0,1]$ is a finite set of rational numbers.
Then, $\Phi(\mathcal R)$ is a set of rational numbers satisfying the \emph{descending chain condition} (\emph{DCC} in short) whose only accumulation point is 1.

\subsection{Divisors} Let $X$ be a normal quasi-projective variety. We say that $D$ is a divisor on $X$ if it is a $\qq$-Weil divisor, i.e., $D$ is a finite sum of prime divisors on $X$ with coefficients in $\qq$.
The \emph{support} of a divisor $D=\sum_{i=1}^n d_iP_i$ is the union of the prime divisors appearing in the formal sum $\mathrm{Supp}(D)= \sum_{i=1}^n P_i$.
Let $f  \colon  X \rar Z$ be a projective morphism of quasi-projective varieties.
Given a divisor $D = \sum d_i P_i$ on $X$, we define
\[
D^v \coloneqq  \sum_{f(D_i) \subsetneq Z} d_i P_i, \; 
D^h \coloneqq  \sum_{f(D_i) = Z} d_i P_i.
\]
We call $D^v$ and $D^h$ the \emph{vertical part} and \emph{horizontal part} of $D$, respectively.
Let $D_1$ and $D_2$ be divisors on $X$.
We write $D_1 \sim_Z D_2$ (respectively $D_1 \sim \subs \qq,Z . D_2$) if there is a Cartier (respectively $\qq$-Cartier) divisor $L$ on $Z$ such that $D_1 - D_2 \sim f^*L$ (respectively $D_1 - D_2 \sim_\qq f^*L$).
Equivalently, we may also write $D_1 \sim D_2$ over $Z$.
The case of $\qq$-linear equivalence is denoted similarly.

\subsection{Pairs}

A \emph{sub-pair} $(X,B)$ is the datum of a normal quasi-projective variety and a divisor $B$ such that $\K X. + B$ is $\qq$-Cartier.
If $B \leq \Supp(B)$, we say that $B$ is a \emph{sub-boundary}, and if in addition $B \geq 0$, we call it \emph{boundary}.
A sub-pair $(X,B)$ is called a \emph{pair} if $B \geq 0$.
A sub-pair $(X,B)$ is \emph{simple normal crossing} (or \emph{log smooth}) if $X$ is smooth, every irreducible component of $\Supp(B)$ is smooth, and \'etale locally $\Supp(B) \subset X$ is isomorphic to the intersection of $r \leq n$ coordinate hyperplanes in $\mathbb{A}^n$.
A \emph{log resolution} of a sub-pair $(X,B)$ is a birational contraction $\pi \colon X' \rar X$ such that $\mathrm{Ex}(\pi)$ is a divisor and $(X',\pi \sups -1._* (\Supp(B)) + \mathrm{Ex}(\pi))$ is log smooth.
Here, $\mathrm{Ex}(\pi) \subset X'$ is the \emph{exceptional set} of $\pi$, i.e., the reduced subscheme of $X'$ consisting of the points where $\pi$ is not an isomoprhism.
If $(X,B)$ is a sub-pair and $f \colon X \rar U$ is a morphism, we say that $(X,B)$ is \emph{log smooth over $U$} if $(X,B)$ is simple normal crossing, and every stratum of $(X,\Supp(B))$, including $X$ itself, is smooth over $U$.

Let $(X_1,B_1)$ and $(X_2,B_2)$ be two pairs.
We say that they are \emph{crepant} to each other if there exist a normal variety $Y$ and birational morphisms $p \colon Y \rar X_1$ and $q \colon Y \rar X_2$ so that $p^*(\K X_1.+B_1)=q^*(\K X_2.+B_2)$.

Let $(X,B)$ be a sub-pair, and let $\pi \colon X' \rar X$ be a birational contraction from a normal variety $X'$.
Then, we can define a sub-pair $(X',B')$ on $X'$ via the identity
\[
\K X'. + B' = \pi^*(\K X. + B),
\]
where we assume that $\pi_* (\K X'.) = \K X.$.
We call $(X',B')$ the \emph{log pull-back} or \emph{trace} of $(X,B)$ on $X'$.
The \emph{log discrepancy} of a prime divisor $E$ on $X'$ with respect to $(X,B)$ is defined as $a_E(X,B) \coloneqq 1 - \mult_E (B')$.
Let $\epsilon$ be a non-negative number.
We say that a sub-pair $(X,B)$ is \emph{$\epsilon$-sub-log canonical} (resp. \emph{$\epsilon$-sub-klt}) if $a_E(X,B) \geq \epsilon$ (resp. $a_E(X,B) > \epsilon$) for every $\pi$ and every $E$ as above.
If $\epsilon=0$, we drop it from the notation.
When $(X,B)$ is a pair, we say that $(X,B)$ is \emph{$\epsilon$-log canonical} or \emph{$\epsilon$-klt}, respectively.
Notice that, if $(X,B)$ is log canonical (resp. klt), we have $\coeff(B) \subset [0,1]$ (resp. $\coeff(B) \subset [0,1)$).

Given a sub-pair $(X,B)$ and an effective $\qq$-Cartier divisor $D$, we define the \emph{log canonical threshold} of $D$ with respect to $(X,B)$ as
\[
\lct (X,B;D) \coloneqq \sup \lbrace t \geq 0 | (X,B+tD) \text{ is sub-log canonical} \rbrace.
\]

\subsection{B-divisors}
Let $X$ be a normal variety, and consider the set of all proper birational morphisms $\pi \colon  X_\pi \rightarrow X$, where $X_\pi$ is normal.
This is a partially ordered set, where $\pi' \geq \pi$ if $\pi'$ factors through $\pi$.
We define the space of {\it Weil b-divisors} as the inverse limit
\begin{equation*}
\mathbf{Div}(X)\coloneqq \varprojlim_\pi \mathrm{Div}(X_\pi),
\end{equation*}
where $\mathrm{Div}(X_\pi)$ denotes the space of Weil divisors on $X_\pi$.
Then, we define the space of {\it $\qq$-Weil b-divisors} $\mathbf{Div}_\qq(X) \coloneqq \mathbf{Div}(X) \otimes \qq$.
In the following, by b-divisor, we will mean a $\qq$-Weil b-divisor.
Equivalently, a b-divisor $\mathbf{D}$ can be described as a (possibly infinite) sum of geometric valuations $V_i$ of $k(X)$ with coefficients in $\qq$,
\[
 \mathbf{D}= \sum_{i \in I} b_i V_i, \; b_i \in \mathbb{\qq},
\]
such that for every normal variety $X'$ birational to $X$, only a finite number of the $V_i$ can be realized by divisors on $X'$. 
The {\it trace} $\mathbf{D}_{X'}$ of $\mathbf{D}$ on $X'$ is defined as 
\[
\mathbf{D}_{X'} \coloneqq \sum_{
\{i \in I \; | \; c_{X'}(V_i)= D_i, \; 
\codim_{X'} (D_i)=1\}} b_i D_i,
\]
where $c_{X'}(V_i)$ denotes the center of the valuation on $X'$.

Given a b-divisor $\mathbf{D}$ over $X$, we say that $\mathbf{D}$ is a {\it b-$\qq$-Cartier} b-divisor if there exists a birational model $X'$ of $X$ such that $\mathbf{D}_{X'}$ is $\qq$-Cartier on $X'$, and for any model $r \colon X''  \rar X'$, we have $\mathbf{D}_{X''} = r^\ast \mathbf{D}_{X'}$.
When this is the case, we will say that $\mathbf{D}$ descends to $X'$ and write $\mathbf{D}= \overline{\mathbf{D}_{X'}}$.
We say that $\mathbf{D}$ is {\it b-effective}, if $\mathbf{D}_{X'}$ is effective for any model $X'$.
We say that $\mathbf{D}$ is {\it b-nef}, if it is b-$\qq$-Cartier and, moreover, there exists a model $X'$ of $X$ such that $\mathbf{D}= \overline{\mathbf{D}_{X'}}$ and $\mathbf{D}_{X'}$ is nef on $X'$. 
The notion of b-nef b-divisor can be extended analogously to the relative case.

\begin{example}
\label{discr.div.ex}
Let $(X, B)$ be a sub-pair.
The \emph{discrepancy b-divisor} $\mathbf{A}(X,B)$ is defined as follows: on a birational model $\pi \colon X' \rar X$, its trace $\mathbf{A}(X,B)_{X'}$ is given by the identity  $\K X'. = \pi^*(\K X. + B) + \mathbf{A}(X,B)_{X'}$.
Then, the b-divisor $\mathbf{A}^*(X,B)$ is defined taking its trace $\mathbf{A}^*(X,B)_{X'}$ on $X'$ to be $\mathbf{A}(X,B)_{X'} \coloneqq \sum \subs a_i > -1. a_i D_i$, where $\mathbf{A}(X,B)_{X'} = \sum_i a_i D_i$.
\end{example}

\subsection{Generalized pairs}
A {\em generalized sub-pair} $(X,B, \mathbf{M})/Z$ over $Z$  is the datum of:
\begin{itemize}
\item a normal variety admitting a projective morphism $X  \rar Z$;
\item a divisor $B$ on $X$;
\item a b-$\qq$-Cartier b-divisor $\mathbf{M}$ over $X$ which descends to a nef$/Z$ $\qq$-Cartier divisor $\mathbf{M}_{X'}$ on some birational model $X' \rightarrow X$.
\end{itemize}
Moreover, we require that $K_X +B+ \mathbf{M}_X$ is $\qq$-Cartier.
If $B$ is effective, we say that $(X,B, \mathbf M)/Z$ is a \emph{generalized pair}.
The divisor $B$ is called the \emph{boundary part} of $(X,B, \mathbf M)/Z$, and $\mathbf M$ is called the \emph{moduli part}.
In the definition, we can replace $X'$ with a higher birational model $X''$ and $\mathbf{M}_{X'}$ with $\mathbf{M}_{X''}$ without changing the generalized pair.
Whenever $\mathbf{M}_{X''}$ descends on $X''$, then the datum of the rational map $X'' \drar X$, $B$, and $\mathbf{M}_{X''}$ encodes all the information of the generalized pair.

Let $(X,B, \mathbf{M})/Z$ be a generalized sub-pair and $\rho\colon Y \rar X$ a projective birational morphism. 
Then, we may write
\[
\K Y.+B_Y + \mathbf{M}_{Y}=\pi^*(K_X+B+\mathbf M _X).
\]
Given a prime divisor $E$ on $Y$, we define the {\em generalized log discrepancy} of $E$ with respect to $(X,B, \mathbf M)/Z$  to be $a_E(X,B,\mathbf M)\coloneqq 1-\mult_{E}(B')$.
If $a_E(X,B,\mathbf M) \geq 0$ for all divisors $E$ over $X$, we say that $(X,B, \mathbf M)/Z$ is \emph{generalized sub-log canonical}.
Similarly, if $a_E(X,B+M) > 0$ for all divisors $E$ over $X$ and $\lfloor B \rfloor \leq 0$, we say that $(X,B, \mathbf M)/Z$ is \emph{generalized sub-klt}.
When $B \geq 0$, we say that $(X,B, \mathbf M)/Z$ is \emph{generalized log canonical} or \emph{generalized klt}, respectively.

\subsection{Canonical bundle formula}\label{section_cbf}

We recall the statement of the {\em canonical bundle formula}. 
We refer to \cite{FG14} for the notation involved and a more detailed discussion about the topic.
Let $(X, B)$ be a sub-pair.
A contraction $f \colon X \rar T$ is an \emph{lc-trivial fibration} if
\begin{itemize}
    \item[(i)] $(X,B)$ is a sub-pair that is sub-log canonical over the generic point of $T$;
    \item[(ii)] $\mathrm{rank} f_* \O X. (\lceil \mathbf{A}^*(X,B)\rceil)=1$, where $\mathbf{A}^*(X,B)$ is the b-divisor defined in Example \ref{discr.div.ex}; and
    \item[(iii)] there exists a $\qq$-Cartier divisor $L_T$ on $T$ such that $\K X. + B \sim_\qq f^* L_T$.
\end{itemize}
Condition (ii) above is automatically satisfied if $B$ is effective over the generic point of $T$.
Given a sub-pair $(X,B)$ and an lc-trivail fibration $f \colon X \rar T$, there exist b-divisors $\mathbf{B}$ and $\mathbf{M}$ over $T$ such that the following linear equivalence relation, known as the {\it canonical bundle formula}, holds
\begin{equation}
    \label{cbf.eqn}
    K_X+B \sim_{\mathbb{Q}} f^\ast(K_T+\mathbf{B}_{T}+\mathbf{M}_{T}).
\end{equation}
For a prime divisor $P \subset T$, its coefficient in $\mathbf{B}_T$ is given by the formula $1-\mathrm{lct}_{\eta_P}(X,B;f^*P)$, where the symbol $\mathrm{lct}_{\eta_P}$ denotes the log canonical threshold over the generic point $\eta_P$ of $P$.
Then, we set $\mathbf{M}_T \coloneqq L_T-(K_T+\mathbf{B}_T)$.
If $f' \colon X' \rar T'$ is a higher model of $f$ with morphisms $\phi \colon X' \rightarrow X$ and $\psi \colon T' \rightarrow T$, one repeats this algorithm with $(X',B')$ and $L_{T'}$, where $K_{X'}+B'=\phi^*(K_X+B)$ and $L_{T'}=\psi^*L_T$.
We refer to \cite{PS09}*{\S~7} for more details.

The b-divisor $\mathbf{B}$ is often called the \emph{boundary part} in the canonical bundle formula; it is a canonically defined b-divisor.
Furthermore, if $B$ is effective, then so is $\mathbf{B}_T$.
The b-divisor $\mathbf{M}$, in turn, is often called the \emph{moduli part} in the canonical bundle formula, and it is in general defined only up to $\mathbb{Q}$-linear equivalence. 
The linear equivalence in \eqref{cbf.eqn} holds at the level of b-divisors: namely, 
\[
\overline{(K_X+B)} \sim_\qq f^*(\mathbf{K}+\mathbf{B}+\mathbf{M}),
\] 
where $\mathbf{K}$ denotes the canonical b-divisor of $T$.
Let $I$ be a positive integer such that $I(\K X. + B) \sim 0$ along the generic fiber of $f$.
Then, by \cite{PS09}*{Construction 7.5}, we may choose $\mathbf M$ in its $\qq$-linear equivalence class so that
\begin{equation*}
I\overline{(K_X+B)} \sim If^*(\mathbf{K}+\mathbf{B}+\mathbf{M}).
\end{equation*} 

The moduli b-divisor $\mathbf{M}$ is expected to detect the variation of the fibers of the morphism $f$.
In this direction, we have the following statement.

\begin{theorem}
\cite{FG14}*{cf. Theorem 3.6} 
\label{classic cbf}
Let $f \colon (X,B) \rar T$ be an lc-trivial fibration and let $\pi \colon T \rar S$ be a projective morphism.
Let $\mathbf{B}$ and $\mathbf{M}$ be the b-divisors that give the boundary and the moduli part, respectively.
Then, $\mathbf K + \mathbf B$ and $\mathbf M$ are b-$\qq$-Cartier b-divisors.
Furthermore, $\mathbf{M}$ is b-nef over $S$.
\end{theorem}

\begin{remark}
In the setup of Theorem \ref{classic cbf}, let $T'$ be a model where the nef part $\mathbf M$ descends in the sense of b-divisors.
Then, $\mathbf M \subs T'.$ is nef over $S$. 
In particular, $(T, \bB T.,\mathbf{M})/S$ is a generalized sub-pair.
\end{remark}

\begin{remark}
In the setup of Theorem \ref{classic cbf}, the b-divisor $\bM.$ is expected to be b-semi-ample.
Furthermore, it is expected that $C\bM.$ is b-free, where $C \in \mathbb Z \subs \geq 1.$ only depends on $\dim(X)$ and the coefficients of $B$ \cite{PS09}*{Conjecture 7.13}.
These facts are known if $\dim(T)=\dim(X)-1$, and we will constantly make use of them in this work \cite{PS09}*{Theorem 8.1}.
\end{remark}

\subsection{Boundedness}\label{sec_bddness}
Here we recall the notion of boundedness for a set of pairs, and we introduce a suitable notion of boundedness for fibrations.

\begin{definition}
Let $\mathfrak{D}$ be a set of projective pairs.
Then, we say that $\mathfrak{D}$ is bounded (resp. birationally bounded) if there exist a pair $(\mathcal{X},\mathcal{B})$, where $\mathcal{B}$ is reduced, and a projective morphism $\pi \colon \mathcal{X} \rar T$, where $T$ is of finite type, such that for every $(X,B) \in \mathfrak{D}$ there are a closed point $t \in T$ and a morphism (resp. a birational map) $f_t \colon \mathcal{X}_t \rar X$ inducing an isomorphism $(X,\Supp(B)) \cong (\mathcal{X}_t,\mathcal{B}_t)$ (resp. such that $\Supp(\mathcal B _t)$ contains the strict transform of $\Supp(B)$ and all the $f_t$ exceptional divisors).
If a set of pairs is birationally bounded and the maps $f_t$ and $f_t \sups -1.$ are isomorphisms in codimension 1, we say that $\mathfrak{D}$ is bounded in codimension 1.
\end{definition}

\begin{definition}\label{def_bdd_fibrations}
Let $\mathfrak{F}$ be a set of fibrations between projective pairs $\phi \colon (X,B) \rar (Y,D)$.
We say that $\mathfrak{F}$ is bounded (resp. birationally bounded) if there exist pairs $(\mathcal{X},\mathcal{B})$, $(\mathcal{Y},\mathcal{D})$, where $\mathcal{B}$ and $\mathcal{D}$ are reduced, a variety of finite type $T$, and projective morphisms
\begin{center}
\begin{tikzcd}
\mathcal{X} \arrow[rd, "\pi"'] \arrow[rr, "\sigma"] &   & \mathcal{Y} \arrow[ld, "\rho"] \\
                                                    & T &                               
\end{tikzcd}
\end{center}
such that
\begin{itemize}
    \item[(i)] the above diagram is commutative; i.e., we have $\pi=\rho \circ \sigma$;
    \item[(ii)] for every $(X,B) \rar (Y,D) \in \mathfrak{F}$, there is a closed point $t \in T$ and morphisms (resp. birational maps) $f_t \colon \mathcal{X}_t \rar X$ and $g_t \colon \mathcal{Y}_t \rar Y$ inducing isomorphisms $(X,\Supp(B)) \cong (\mathcal{X}_t,\mathcal{B}_t)$ (resp. such that $\Supp(\mathcal B _t)$ contains the strict transform of $\Supp(B)$ and all the $f_t$ exceptional divisors) and $(Y,\Supp(D)) \cong (\mathcal{Y}_t,\mathcal{D}_t)$ (resp. such that $\Supp(\mathcal D _t)$ contains the strict transform of $\Supp(D)$ and all the $g_t$ exceptional divisors); and
    \item[(iii)] for every $(X,B) \rar (Y,D) \in \mathfrak{F}$ and $t$ as in condition (ii), we have $\phi \circ f_t = g_t \circ \sigma_t$ (resp. as rational maps),
    where $\sigma_t$ denotes the restriction of $\sigma$ to $\mathcal{X}_t$.
\end{itemize}
If a set of fibrations is birationally bounded and the maps $f_t$, $f_t \sups -1.$, $g_t$ and $g_t \sups -1.$ are isomorphisms in codimension 1, we say that $\mathfrak{F}$ is bounded in codimension 1.
\end{definition}

\subsection{The geometric Tate--Shafarevich group}
Here, we recall a few facts about the geometric Tate--Shafarevich group.
We will limit ourselves ot state only those facts regarding the Tate--Shafarevich group that will be needed in the article.
We refer to \cites{DG94,Gro94} for a detailed development of the theory we need.
An introduction to the topic over a one-dimensional base can be found in \cite{Sha65}.
Finally, all the needed facts about \'etale cohomology and group cohomology can be found in \cites{Mil80,Mil06}.

Let $Y$ be a variety defined over a field of characteristic zero, and let $k(Y)$ be its field of fractions.
Let $\pi \colon X \rar Y$ be an elliptic fibration, that is, a contraction whose general fiber is an elliptic curve.
Then, the generic fiber $X_{{\eta_Y}}$ is a genus 1 curve over $k(Y)$ with possibly no $k(Y)$-rational points.
Let $J(X)_{{\eta_Y}}$ denote its Jacobian.
Then, $X_{{\eta_Y}}$ is a principal homogeneous space over $J(X)_{{\eta_Y}}$ defined over $k(Y)$.
Furthermore, we can consider a projective model $\rho \colon J(X) \rar Y$ realizing $J(X)_{{\eta_Y}}$ over $Y$.

For some finite Galois extension $K \colon k(Y)$, the variety $X_{{\eta_Y}}$ acquires a $K$-rational point.
In particular, $X_{{\eta_Y}}$ and $J(X)_{{\eta_Y}}$ are non-canonically isomorphic over $K$.
From a geometric point of view, it means that we can find a finite Galois morphism $Y' \rar Y$ such that $X \times_Y Y'$ admits a rational section.

Now, let $E$ be an elliptic curve defined over $k(Y)$, and let $L \colon k(Y)$ be a finite Galois extension.
We want to consider all the elliptic fibrations $\pi \colon X \rar Y$ so that $X_{{\eta_Y}}$ is a principal homogeneous space over $E$ admitting an $L$-rational point.
That is, we are interested in elliptic fibrations $\pi \colon X \rar Y$ that acquire a rational section after the base change $\mathrm{Spec}(L) \rar Y$ and whose associated Jacobian fibration has $E$ as generic fiber.
Said otherwise, $\pi$ is an elliptic fibration whose generic fiber becomes isomorphic to $E_L \coloneqq E \otimes_{k(Y)} L$ as $L$-schemes.
This set is parametrized by the cohomology group $H^1(\mathrm{Gal}(L \colon k(Y)),E(L))$.
More precisely, this group parametrizes the isomorphism classes of the generic fibers $X_{{\eta_Y}}$ so that $\mathrm{Pic}^0(X_{{\eta_Y}}) \simeq E$ and $X_{{\eta_Y}}$ has an $L$-rational point.
Equivalently, the group $H^1(\mathrm{Gal}(L \colon k(Y)),E(L))$ parametrizes the birational classes of the elliptic fibrations with base $Y$ and associated Jacobian fibration corresponding to $E$ that acquire a rational section after the base change $\mathrm{Spec}(L) \rar Y$.
Notice that, by standard properties of group cohomology, the order of the elements of $H^1(\mathrm{Gal}(L \colon k(Y)),E(L))$ is finite and divides $[L \colon k(Y)]$.

Now, if we want to consider all the principal homogeneous spaces over $E$, we need to consider all possible finite Galois extensions $L \colon k(Y)$.
This information is encoded in the Weil--Ch\^atelet group $WC(E) \coloneqq H^1(\mathrm{Gal}(\overline{k(Y)}\colon k(Y)),E(\overline{k(Y)}))$, which is the direct limit of all possible groups $H^1(\mathrm{Gal}(L \colon k(Y)),E(L))$ as above.
In particular, we have the following geometric consequence.

\begin{lemma} \label{lemma order WC}
Let $E$ be an elliptic curve defined over $k(Y)$, and let $\pi \colon X \rar Y$ be an elliptic fibration with $X_{{\eta_Y}} \in WC(E)$.
If $\pi$ admits a multisection of degree $d$, then the order of $X_{{\eta_Y}}$ in $WC(E)$ divides $d!$.
\end{lemma}

\begin{proof}
Let $\pi \colon X \rar Y$ be as in the statement.
Then, the multisection induces a field extension $L$ of degree $d$ so that $X_{{\eta_Y}}$ admits an $L$-rational point.
Notice that $L \colon k(Y)$ may not be a Galois extension.
Thus, there is a Galois extension $L' \colon k(Y)$ so that $X_{{\eta_Y}}$ admits an $L'$-rational point, and $[L'\colon k(Y)]$ divides $d!$.
Then, the claim follows, as $X_{{\eta_Y}}$ is an element of $H^1(\mathrm{Gal}(L' \colon k(Y)),E(L'))$.
\end{proof}

The Weil--Ch\^atelet group $WC(E)$ is a very large group that parametrizes all the birational classes of all elliptic fibrations $\pi \colon X \rar Y$ whose Jacobian fibration is a compactification of $E$ over $Y$.
That is, $WC(E)$ parametrizes all birational classes of elliptic fibrations $\pi \colon X \rightarrow Y$ whose geometric generic fiber is isomorphic to $E_{\overline{k(Y)}}$ as $\overline{k(Y)}$-schemes.
This group admits a natural subgroup, called the Tate--Shafarevich group, which is denoted by $\Sh_Y(E)$.
For a formal definition of $\Sh_Y(E)$, we refer to \cite{DG94}.
Here, we limit ourselves to the following characterization as a set:
\[
\Sh_Y(E) = \lbrace C \in WC(E) | X_C \rar Y \; \text{has a rational section \'etale locally at $y$ for every point}\;y \in Y \rbrace,
\]
where $X_C \rar Y$ is some proper model of the curve $C$ defined over $k(Y)$ \cite{Gro94}*{\S~3}.
Thus, $\Sh_Y(E)$ imposes pretty restrictive conditions on the type of singular fibers that can occur.
In particular, multiple fibers cannot occur over codimension 1 points of the base.

\begin{proposition} \label{lemma tecnico DG}
Let $f \colon X \rar Y$ be an elliptic fibration, and let $j \colon J(X) \rar Y$ be its associated Jacobian fibration.
Let $p \in Y$ be a closed point.
Assume that $Y$ is smooth, $\dim(Y) \geq 2$, $f$ is smooth over $Y \setminus \lbrace p \rbrace$, and $j$ is smooth with a regular section.
Then, $X \in \Sh_Y(J(X)_{{\eta_Y}})$.
\end{proposition}

\begin{proof}
Let $\O Y,\overline{p}.$ denote the strict henselization of $\O Y,p.$, and let $\overline{p}$ denote the closed point of $\mathrm{Spec}(\O Y,\overline{p}.)$.
Let $\overline{S}$ denote $\mathrm{Spec}(\O Y,\overline{p}.)$ and let $S \coloneqq \overline{S} \setminus \lbrace \overline{p} \rbrace$.
Let $A$ denote the generic fiber of $J(X) \times_Y \overline{S}$.
Then, as $\overline{S}$ is strictly local, we have $\Sh_{\overline{S}}(A)=0$.
Indeed, $\O Y,\overline{p}.$ coincides with its own strict henselization, and thus the description of $\Sh_{\overline{S}}(A)$ as an intersection of kernels necessarily returns 0; see \cite{DG94}*{\S~1}.
If $f \colon X \rar Y$ does not correspond to an element of $\Sh_Y(J(X)_{{\eta_Y}})$, then it induces a non-zero element of $\Sh_S(A)$.
Thus, to conclude, it is enough to show $\Sh_S(A)=0$.
For this type of argument, see for instance \cite{DG94}*{\S~3}.
\newline
By assumption, $J(X) \times_Y \overline{S} \rar \overline{S}$ satisfies the conditions of \cite{DG94}*{Theorem 3.1}.
Thus, following the notation of \cite{DG94}*{Theorem 3.1}, $\Sh_S(A)$ is a subgroup of the group $H^2(S,\mathcal{E})$.
Here, $\mathcal{E}$ is the \'etale sheaf defined in \cite{DG94}*{Definition 1.8}.
Since $j$ is smooth, all the fibers of $J(X) \times_Y \overline{S}$ are geometrically integral.
Thus, by \cite{DG94}*{Proposition 1.12}, $\mathcal{E}=0$.
Then, the claim follows.
\end{proof}

\subsection{Technical statements} Here we collect the technical statements that will be needed in the course of the main proofs.

\begin{proposition} \label{small q-fact}
Let $(X,\Delta)$ be a klt pair, and let $f \colon X \rar Y$ be a contraction with $\K X. + \Delta \sim \subs \qq,f. 0$.
Let $\pi \colon Y' \rar Y$ be a small morphism.
Then, there exist a $\qq$-factorial pair $(X',\Delta')$ admitting a contraction $g \colon X' \rar Y'$ so that $X$ and $X'$ are isomorphic in codimension 1 and $(X',\Delta')$ is crepant to $(X,\Delta)$.
\end{proposition}

\begin{proof}
Let $\phi \colon \hat X \rar X$ be a log resolution of $(X,\Delta)$ admitting a morphism $\hat X \rar Y'$.
Since $(X,\Delta)$ is klt, it is $\epsilon$-log canonical for some $\epsilon > 0$.
Then, let $E$ be the reduced $\phi$-exceptional divisor, and define $\Gamma \coloneqq (1-c)E$, where $c$ is a rational number satisfying $0 < c < \epsilon$.
Let $X^\nu$ denote the normalization of the main component of $X \times_Y Y'$.
First, we run a partial $(\K \hat X. + \phi_* \sups -1. \Delta + \Gamma)$-MMP over $X^\nu$.
By \cite{Fuj11}*{Theorem 2.3}, after finitely many steps, this MMP contracts the prime components of $\Gamma$ that are exceptional over $X^\nu$.
In particular, all the prime components of $\Gamma$ that dominate $Y'$ are contracted, and all the prime components of $\Gamma$ that are not contracted have a center of codimension at least 2 in $Y'$.
Thus, by \cite{Lai11}*{Lemma 2.10}, we can contract these leftover components by running an MMP with scaling relative to $Y'$.
The model thus obtained satisfies the properties in the statement.
\end{proof}

\begin{lemma} \label{lemma restrict cbf}
Let $(\mathcal{X},\mathcal{B})$ be a sub-pair admitting a tower of contractions $\mathcal X \rar \mathcal Y \rar T$.
Assume that the following conditions are satisfied:
\begin{itemize}
    \item $(\mathcal X, \mathcal B) \rar \mathcal Y$ is an lc-trivial fibration, inducing a generalized pair $(\mathcal Y, \mathcal B _\mathcal{Y},\bM.)$ on $\mathcal{Y}$;
    \item $\mathcal{X}$, $\mathcal{Y}$ and $T$ are smooth;
    \item $\bM.$ descends on $\mathcal{Y}$, and $(\mathcal{Y},\mathcal{B}_\mathcal{Y})$ is log smooth over $T$;
    \item $(\mathcal{X},\mathcal{B})$ is log smooth over $T$.
\end{itemize}
Let $t \in T$ be a closed point, and assume that $(\mathcal{X}_t,\mathcal{B}_t) \rar \mathcal{Y}_t$ is an lc-trivial fibration.
Then, the boundary and moduli b-divisors induced on $\mathcal{Y}_t$ by $(\mathcal{X}_t,\mathcal{B}_t)$ agree with the restrictions of $\mathcal{B}_\mathcal{Y}$ and $\bM.$, respectively.
\end{lemma}

We recall that, given a sub-pair $(W,\Delta)$ and a morphism $\pi \colon W \rightarrow U$, we say that $(W,\Delta)$ is {\it log smooth over $U$} if $W$ is smooth, $\Supp(\Delta)$ is a divisor with simple normal crossings, and every stratum of $(W,\Supp(\Delta))$ (including $W$) is smooth over $U$.

\begin{remark}
In the setup of Lemma \ref{lemma restrict cbf}, for some special point $t$, the induced fibration $(\mathcal{X}_t,\mathcal{B}_t) \rightarrow \mathcal{T}_t$ may not satisfy condition (ii) in \S~\ref{section_cbf}.
Nevertheless, one can apply the construction of boundary and moduli divisors to every fiber as described in \S~\ref{section_cbf}.
\end{remark}

\begin{proof}
Fix $t \in T$ as in the statement, and let $H \subset T$ be a general smooth divisor through $t$.
Then, as $H$ is general away from $t \in T$, $(\mathcal{X},\mathcal{B}) \times_T H \rar \mathcal{Y} \times_T H$ is an lc-trivial fibration.
By induction on the dimension of $T$, it suffices to show that this fibration has the same properties as $(\mathcal{X},\mathcal{B}) \rar \mathcal{Y} \rar T$ and that the boundary and moduli b-divisors induced on $\mathcal{Y} \times_T H$ are the restrictions of $\mathcal{B}_\mathcal{Y}$ and $\bM.$, respectively.
\newline
By base change, the properties in the statement are satisfied by $(\mathcal{X},\mathcal{B}) \times_T H \rar \mathcal{Y} \times_T H$.
Thus, we are left with showing the statement about the b-divisors.
Since $(\mathcal{Y},\mathcal{B}_\mathcal{Y})$ is log smooth over $T$, every fiber of $\mathcal{Y} \rar T$ intersects properly and transversally every stratum of $\Supp(\mathcal{B}_\mathcal{Y})$.
Furthermore, by \cite{Har77}*{Proposition III.10.1}, $H \times_T \mathcal{Y}$ is smooth.
Thus, $H \times_T \mathcal{Y}$ satisfies properties (1), (2) and (3) in the proof of \cite{Flo14}*{Lemma 3.1}.
Then, by the proof of \cite{Flo14}*{Lemma 3.1}, the claim about b-divisors on the model $\mathcal{Y} \times_T H$ follows.
By construction, the standard normal crossing assumptions defined in \cite{Kol07}*{Definition 8.3.6} are satisfied.
So, by \cite{Kol07}*{Theorem 8.3.7}, the moduli b-divisor induced by $(\mathcal{X},\mathcal{B})\times_T H \rar \mathcal{Y}\times_T H$ descends onto $\mathcal{Y} \times_T H$.
Thus, the statement holds at the level of b-divisors, and not just on a specific model.
This concludes the proof.
\end{proof}

\begin{corollary} \label{stratify cbf}
Let $(\mathcal{X},\mathcal{B})$ be a pair admitting a tower of contractions $\mathcal{X} \rar \mathcal{Y} \rar T$.
Assume that $(\mathcal{X},\mathcal{B}) \rar \mathcal{Y}$ is an lc-trivial fibration, $(\mathcal{X},\mathcal{B}) \rar T$ is a family of pairs, and $\mathcal{Y} \rar T$ is a family of normal varieties.
Then, there exists a stratification of $T$ such that the following holds.
Let $(\mathcal{Y},\mathcal{B}_\mathcal{Y},\bM.)$ be the generalized pair induced by the canonical bundle formula, and let $t \in T$ be a closed point.
Then, the boundary and moduli b-divisors induced on $\mathcal{Y}_t$ by $(\mathcal{X}_t,\mathcal{B}_t)$ agree with the restrictions of $\mathcal{B}_\mathcal{Y}$ and $\bM.$, respectively.
\end{corollary}

\begin{remark}
In the statement of Corollary \ref{stratify cbf} we are abusing notation in the following sense: we replace $T$ with a stratification, and we replace $\mathcal{X}$ and $\mathcal{Y}$ with the corresponding stratifications induced by base change.
\end{remark}

\begin{proof}
Up to a first stratification, we may assume that $T$ is smooth.
Now, fix an irreducible component $T_i$ of $T$, and let $(\mathcal{X}_i,\mathcal{B}_i)$ and $\mathcal{Y}_i$ denote the corresponding irreducible components.
Let $\mathcal{Y}'_i$ be a log resolution of $(\mathcal{Y}'_i,\mathcal{B}\subs \mathcal{Y}'_i.)$ where $\bM i.$ descends.
Then, let $(\mathcal{X}'_i,\mathcal{B}'_i)$ be a log resolution of $(\mathcal{X}_i,\mathcal{B}_i)$ that factors through $\mathcal{Y}'_i$.
Then, by generic smoothness, the assumptions of Lemma \ref{lemma restrict cbf} are satisfied over a non-empty open subset $U_i \subset T_i$.
As the b-divisors corresponding to a fiber $t \in U_i$ can be constructed both considering $(\mathcal{X}'_{i,t},\mathcal{B}'_{i,t}) \rar \mathcal{Y}'_{i,t}$ or $(\mathcal{X}_{i,t},\mathcal{B}_{i,t}) \rar \mathcal{Y}_{i,t}$, the claim of the statement holds for $t \in U_i$.
By Noetherian induction on $T_i \setminus U_i$, we conclude that the statement holds over $T_i$.
As there are finitely many $T_i$'s, then the claim follows.
\end{proof}

\begin{lemma} \label{lemma cbf jacobian}
Let $(X,\Delta)$ by a klt $n$-fold admitting an elliptic fibration $f \colon X \rar Y$ with $\K X. + \Delta \sim \subs \qq,f. 0$.
Let $(Y,B_Y,\bM.)$ denote the generalized pair induced on $Y$.
Let $J(X) \rar Y$ be a relatively minimal terminal model for the Jacobian fibration, and let $\pi \colon Y' \rar Y$ be the relative ample model of $J(X)$.
Then, the trace of $(Y,B_Y,\bM.)$ on $Y'$ is a generalized pair, i.e., the trace of the boundary b-divisor on $Y'$ is effective.
\end{lemma}

\begin{proof}
Up to taking a small $\qq$-factorialization, we may assume that $X$ is $\qq$-factorial.
Then, let $\hat X$ be a relative minimal model for $\K X.$ over $Y$.
Then, $X \drar \hat X$ is a birational contraction, and the relative ample model $\hat Y \rar Y$ defines a birational morphism.
Let $(\hat X,\hat \Delta)$ denote the trace of $(X,\Delta)$ on $\hat X$, and let $(\hat Y,B_{\hat Y},\bM.)$ denote the trace of $(Y,B_Y,\bM.)$ on $\hat Y$.
Let $(\hat Y, \Omega_{\hat Y},\bM.)$ the generalized pair induced by $(\hat X , 0)$ on $\hat Y$.
Since we have $\K \hat X. \sim \subs \qq,\hat Y . 0$ and $\hat \Delta \geq 0$,
we have $0 \leq \Omega_{\hat Y} \leq B_{\hat Y}$.
In particular, $B_{\hat Y}$ is effective.
Hence, for the purposes of the proof, we may replace $(X,\Delta)$ with $(\hat X,\hat \Delta)$.
Indeed, the ample model of the Jacobian fibration relative to $\hat Y$ naturally admits a rational contraction to $Y'$, and the coefficients of the divisors appearing on $Y'$ agree with the coefficients on $\hat Y$.
This follows from the fact that a minimal model of the Jacobian fibration relative to $\hat Y$ admits a birational contraction to a minimal model relative to $Y$.
Thus, in the rest of the proof, we may assume that $X$ is $\qq$-factorial and $\K X. \sim_{\qq,f}0$.
\newline
By Proposition \ref{small q-fact}, we may assume that $Y$ is $\qq$-factorial.
Let $U \subset Y$ be a big open subset such that
\begin{itemize}
    \item $\Supp(B_Y)$ is simple normal crossing on $U$;
    \item $\K X. \sim_\qq 0$ over $U$; and
    \item $\K J(X). \sim_\qq 0$ over $U$.
\end{itemize}
This open set exists, as $J(X)$ is relatively minimal over $Y$, and its relative ample model is birational to $Y$.
Let $B_U$ denote the restriction of $B_Y$ to $U$, and let $\Gamma_U$ denote the boundary divisor of the canonical bundle formula for $J(X) \times_Y U \rar U$.
Then, by \cite{Gro94}*{Lemma 1.6}, we have $0 \leq \Gamma_U \leq B_U$.
Now, let $(Y',\Gamma,\bM.)$ denote the generalized pair induced by $J(X) \rar Y'$.
Notice that the moduli b-divisor is the same as the one corresponding to the morphism $(X,\Delta) \rar Y$, as the corresponding $j$-maps agree.
We have $\Gamma \geq 0$.
Furthermore, as $Y'$ is the relative ample model of $\K J(X).$ over $Y$, $\K Y'. + \Gamma + \bM Y'.$ is ample over $Y$.
Since $Y$ is $\qq$-factorial, we can consider the generalized pair $(Y,\pi_*\Gamma,\bM.)$.
Let $\Gamma'$ denote the trace of the corresponding boundary b-divisor on $Y'$.
Thus, the divisor $\K Y'. + \Gamma + \bM Y'.-\pi^*(\K Y. + \pi_*\Gamma + \bM Y.)$ is $\pi$-ample, and its support is $\pi$-exceptional.
Then, by the negativity lemma \cite{KM98}*{Lemma 3.39}, we have
\[
\K Y'. + \Gamma + \bM Y'. \leq \K Y. + \Gamma' + \bM Y'..
\]
Since $(\pi_* \Gamma)|_U=\Gamma_U$, we have $\pi_*\Gamma \leq B_Y$.
Then, the claim follows.
\end{proof}

\begin{lemma} \label{lemma same smooth fibers 1}
Let $(X_1,\Delta_1)$ and $(X_2,\Delta_2)$ be two quasi-projective $\qq$-factorial klt pairs admitting contractions $f_1 \colon X_1 \rar Y$ and $f_2 \colon X_2 \rar Y$ to the same smooth quasi-projective variety $Y$ with $\dim(Y)=\dim(X)-1$.
Further assume that $\K X_i. + \Delta_i \sim \subs \qq,f_i. 0$ for $i=1,2$, $X_1$ and $X_2$ are isomorphic in codimension 1, and $(X_1,\Delta_1)$ and $(X_2,\Delta_2)$ are crepant to each other.
Fix $y \in Y$.
Then, $f_1$ is smooth over $y$ if and only if so is $f_2$.
\end{lemma}

\begin{proof}
It suffices to show that $X_1$ and $X_2$ are connected by a sequence of $(\K X_i. + \Delta_i)$-flops over $Y$.
By the assumption $\dim(Y)=\dim(X)-1$, a smooth fiber is a curve in $X$ that deforms over an open set of $Y$.
Thus, a smooth fiber cannot be contained in the flopped locus.
Now, let $A_1$ be an ample irreducible divisor on $X_1$, and let $A_2$ be its strict transform on $X_2$.
Then, $A_2$ is big and movable on $X_2$.
For $0 < \epsilon \ll 1$, the pair $(X_2,\Delta_2+\epsilon A_2)$ is klt.
Then, we can run a relative $(\K X_2. + \Delta_2+\epsilon A_2)$-MMP with scaling over $Y$.
Since $\K X_2. + \Delta_2 \sim \subs \qq,f_2. 0$, it is an $A_2$-MMP.
Since $A_2$ is movable over $Y$, it has to be a sequence of $A_2$-flips.
By \cite{BCHM}, this MMP terminates with a good relative minimal model.
By construction, this model is $\qq$-factorial, it isomorphic in codimension 1 to $X_1$, and it admits a morphism to $X_1$, since $X_1$ is the relative ample model of $A_1$.
Since $X_1$ is $\qq$-factorial, this morphism cannot be a small morphism.
Therefore, this minimal model has to be $(X_1,\Delta_1+\epsilon A_1)$.
This concludes the proof.
\end{proof}

\begin{lemma} \label{lemma same smooth fibers 2}
Let $(X,\Delta)$ be a klt pair admitting a contraction $f \colon X \rar Y$ of relative dimension 1, where $Y$ is smooth.
Let $\pi \colon X' \rar X$ be a small morphism, and let $(X',\Delta')$ denote the trace of $(X,\Delta)$ on $X'$.
Fix $y \in Y$.
Then, $X \rar Y$ is smooth over $y$ if and only if so is $X' \rar Y$.
\end{lemma}

\begin{proof}
Since $Y$ is smooth, if $X \rar Y$ is smooth over $y \in Y$, then $X$ is smooth over a neighborhood of $y \in Y$.
In particular, $X$ does not admit small modifications over $y$.
Now, assume that $X' \rar X$ is not an isomorphism over $y \in Y$.
Then, the fiber $X'_y$ is not an irreducible curve, as $X'_y \rar X_y$ is not an isomorphism.
This concludes the proof.
\end{proof}

\section{Examples}

In this work, we are interested in $n$-folds $(X,\Delta)$ with $\kappa(X,\Delta)=n-1$.
Among these, one may be particularly interested in smooth varieties without boundary and their minimal models.
If $n=2$, we get smooth elliptic surfaces with $\kappa(X)=1$.
If a set $\mathfrak{Q}$ of elliptic surfaces with Kodaira dimension 1 is bounded, then the minimal degree of a multisection of the Iitaka fibration for $X \in \mathfrak{Q}$ is bounded.
Indeed, let $\pi \colon \mathcal{X} \rar T$ be a bounding family.
Up to stratification, we may assume that $\pi$ is smooth.
Thus, by deformation invariance of plurigenera \cite{wilson}, the relative Iitaka fibration of $\pi$ induces the Iitaka fibration fiberwise.
Thus, $\pi$ factors as an ellitpic fibration over $T$, say $\mathcal{X} \rar \mathcal{C} \rar T$.
For simplicity, assume that $T$ is irreducible (in general, it has finitely many irreducible components).
Then, a multisection of degree $k$ of $\mathcal{X} \rar \mathcal{C}$ induces a multisection of degree $k$ of $\mathcal{X}_t \rar \mathcal{C}_t$ for $t \in T$ general.
Thus, by Noetherian induction, there is a positive integer $d$ such that every $X \in \mathfrak{Q}$ admits a multisection of the Iitaka fibration of degree at most $d$.
More generally, if we have a family of elliptic $n$-folds that are bounded and such that the elliptic fibration deforms in the family, we have an upper bound on the minimal degree of a multisection.

Therefore, the hypotheses of Theorem \ref{main thm} are necessary to achieve boundedness.
On the other hand, one may wonder whether the boundedness of the base of the Iitaka fibration may put any constraints on the minimal degree of a multisection.
We will show that this is not the case.
We will address the isotrivial and the non-isotrivial cases separately.

\begin{example} \label{example1}
Let $C$ be a Riemann surface of genus $g(C) \geq 2$.
Denote by $\pi _n \colon C_n \rar C$ a cyclic cover with group $\zz / n \zz$.
By choosing an appropriate $n$-torsion element in $\mathrm{Pic}^0(C)$, we may assume that $\pi_n \colon C_n \rar C$ is \'etale.
Let $E$ be an elliptic curve, and let $p_n \in E$ be a distinguished point of order $n$ on $E$.
If $E$ is chosen to be very general, we have $\mathrm{Hom}(E,J(C_n))=0$ for all $n$.
In particular, $\Pic (C_n \times E) \simeq \Pic(C_n) \times \Pic(E)$.
Let $\zz / n \zz$ act on $E$ by translation by $p_n$.
Therefore, $\zz / n \zz$ acts diagonally on $C_n \times E$.
Let $\sigma_n \colon C_n \times E \rar S_n$ denote the quotient by this action.
Thus, we get a smooth and isotrivial fibration $f_n \colon S_n \rar C$ with fiber $E$.
By construction, we have $\K S_n. \sim_{\qq,f_n} 0$.
Since $f_n$ is smooth and isotrivial, it follows from Kodaira's canonical bundle formula that we have $\K S_n. \sim_\qq f_n^* \K C.$.
In particular, $f_n$ is the Iitaka fibration of the surface $S_n$.
\newline
Now, a $k$-section of $f_n$ corresponds to a $\zz/n\zz$-invariant $k$-section of $\sigma_n$.
If a divisor $D \subset C_n \times E$ is $\zz / n \zz$-invariant, then so is $\O C_n \times E. (D)$ in $\Pic (C_n \times E)$.
By our choice of $E$, we have $D \sim P \boxtimes Q$ for some divisors $P$ and $Q$ on $C_n$ and $E$, respectively.
Since the action preserves the two natural projections on $C_n \times E$, it follows that the line bundles $\O C_n. (P)$ and $\O E. (Q)$ have a $\zz / n \zz$-action.
Thus, $\O E. (Q)$ is the pull-back of a line bundle under the quotient morphism $E \rar E/(\zz / n \zz) \cong E$.
In particular, we have $n| \deg(Q)$.
As $D$ is a $k$-section, $Q$ has has degree $k > 0$.
Then,
we have that $k \geq n$.
Thus, the set of surfaces $\lbrace S_n \rbrace \subs n \in \nn.$ is not bounded, as the minimal degree of a multisection of $f_n$ is not bounded.
\end{example}

\begin{example} \label{example2}
Let $B$ be a curve, and let $J \rar B$ be a Jacobian fibration that is not isotrivial.
Further, assume that the Tate--Shafarevich $\Sh_B(J_\eta)$ group is infinite.
This can be achieved by \cite{Sha65}*{Theorem VII.11}.
For instance, this can be achieved by an elliptic K3 surface, as observed in \cite{Gro94}.
For an example of surface of Kodaira dimension 1, one could consider a base curve $B$ with genus at least 2 and the elliptic surface induced by base change from an elliptic K3 surface.
Then, we claim that there is a sequence of surfaces $\pi_n \colon V_n \rar B$ that are locally \'etale isomorphic to $J$ (i.e., they are accounted for in $\Sh_B(J_\eta)$) and such that the minimal degree of a multisection of $\pi_n$ diverges to $\infty$.
\newline
Let $\pi \colon V \rar B$ be a fibration that is locally isomorphic to $J \rar B$, and let $d$ be the minimal degree of a multisection of $\pi$.
Then, there is a degree $d$ extension $K : k(B)$ such that $V_\eta$ has a $K$-rational point.
This extension may not be Galois, but we can find a further extension $L : K$ such that $L: k(B)$ is Galois of degree at most $d!$.
Then, $V_\eta$ corresponds to a non-zero element $\sigma$ of $H^1(G,E(L))$, where we set $G \coloneqq \mathrm{Gal}(L:k(B))$.
Thus, the order of $\sigma$ divides $d!$.
Therefore, the order of the image of $\sigma$ in the Weil--Ch\^atelet group $WC(J_\eta)$ divides $d!$.
Call this element $\tau$.
By assumption, $\tau$ is an element of $\Sh_B(J_\eta)$, which is a subgroup of $WC(J_\eta)$.
On the other hand, by the choice of $J \rar B$, the Tate--Shafarevich group contains a non-trivial divisible group.
Therefore, for every $n \in \nn$ we can find an element $\rho_n$ of order $n$ in it.
Let $\pi_n \colon V_n \rar B$ the minimal elliptic fibration corresponding to $\rho_n$.
Then, by the above discussion, $\limsup d_n = \infty$, where $d_n$ denotes the minimal degree of a multisection of $\pi_n$.
\end{example}

Let $S$ be a minimal surface with $\kappa (S) =1$, and let $\pi \colon S \rar C$ be the Iitaka fibration.
By Kodaira's canonical bundle formula, we may write $\K S. \sim_\qq \pi^*(\K C. + B_C + M_C)$, where $B_C \geq 0$ and $M_C$ can be represented by an effective divisor.
Example \ref{example1} and Example \ref{example2} show that it is not enough to fix $\deg(\K C. + B_C + M_C)$ to achieve boundedness of surfaces of Kodaira dimension 1.
On the other hand, the discussion above shows that the minimal degree of a multisection is bounded in a family of elliptic fibrations.
Therefore, the assumptions of Theorem \ref{main thm} are both necessary and sufficient to achieve boundedness (in codimension 1).

\section{Good minimal models}

In this section, we prove the existence of good minimal models under some conditions on the Kodaira dimension.
The statement generalizes results 
of Grassi and Wen for $n$-folds with no boundary of Kodaira dimension $n-1$ \cite{GW19}, as well as results of Hao and Schreieder for smooth $n$-folds of Kodaira dimension at least $n-3$ \cite{HS20}*{Theorem 2.1}.

\begin{proof}[Proof of Theorem \ref{good mm}]
Let $X \drar Y$ be the Iitaka fibration of the pair $(X,\Delta)$, where we have
$$
Y=\mathrm{Proj}(\oplus_{l \geq 0}\Gamma(X,\mathcal{O}_X(l(K_X+\Delta)))).
$$
Notice that the ring $\oplus_{l \geq 0}\Gamma(X,\mathcal{O}_X(l(K_X+\Delta)))$ is finitely generated by \cite{BCHM}, so the $\mathrm{Proj}$ construction is well defined and provides a normal variety.
Let $\pi \colon X' \rar X$ be a birational morphism so that $f \colon X' \rar Y$ is a morphism.
We may assume that $\pi$ is a log resolution of $(X,\Delta)$.
Let $E'$ denote the reduced $\pi$-exceptional divisor, and set $\Delta' \coloneqq \pi_*^{-1}\Delta$.
Then, as $(X,\Delta)$ is klt, we may find $0<c \ll 1$ so that $(X',\Delta'+(1-c)E')$ is klt and has the same log canonical ring as $(X,\Delta)$.
Furthermore, we may assume that $\mathrm{mld}(X,\Delta) > c$.
This implies that every prime component of $E'$ lies in the stable base locus of $K_{X'}+\Delta'+(1-c)E'$.
\newline
As aruged in the proof of \cite{FM00}*{Theorem 5.2}, the general fiber of $f \colon (X',\Delta'+(1-c)E') \rar Y$ is a klt pair of Kodaira dimension 0.
Furthermore, by assumption, the dimension of the general fiber is at most 3.
Hence, by \cite{KMM94}, the general fiber of $f$ has a good minimal model.
Therefore, by generic smoothness, \cite{HMX18b}*{Theorem 1.9.1}, and \cite{HX13}*{Theorem 1.1}, it follows that $(X',\Delta'+(1-c)E')$ has a relative good minimal model over $Y$.
Call it $(X'',\Delta''+(1-c)E'')$, and let $Y''$ be the corresponding ample model.
By construction, $Y'' \rar Y$ is a birational morphism.
\newline
Since $\K X''.+\Delta''+(1-c)E''\sim_{\qq,Y''} 0$, we may apply the canonical bundle formula.
Thus, a generalized klt pair $(Y'',B'',\bM.)$ is induced on $Y''$.
Since $\K Y''.+B''+\bM Y''.$ is big, by \cite{BZ16}*{proof of 4.4.(2)}, we may run a $(\K Y''.+B''+\bM Y''.)$-MMP with scaling, which terminates with a good minimal model.
Let $(Y^m,B^m,\bM.)$ denote the model thus obtained.
By \cite{HX13}*{Corollary 2.13}, we can follow this MMP step by step on $X''$.
In particular, there is a birational contraction $X'' \drar X^m$ such that $\K X^m. + \Delta^m+(1-c)E^m\sim_\qq g^*(\K Y^m. + B^m +\bM Y^m.)$, where $\Delta^m$ and $E^m$ denote the push-forwards of $\Delta''$ and $E''$ on $X^m$, repsectively.
Furthermore, this contraction preserves the log canonical ring of $(X'',\Delta''+(1-c)E'')$.
Since $\K Y^m. + B^m +\bM Y^m.$ is semi-ample, it follows that $\K X^m.+\Delta^m+(1-c)E^m$ is semi-ample.
On the other hand, since $E'$ is in the stable base locus of $(X',\Delta'+(1-c)E')$, it follows that $E^m=0$.
In particular, $X \drar X^m$ is a birational contraction.
\newline
To conclude, we need to show that $X \drar X^m$ is $(\K X. + \Delta)$-negative.
First, we notice that, since $\K X^m. + \Delta^m$ is semi-ample and $(X^m,\Delta^m)$ has the same log canonical ring as $(X,\Delta)$, it follows that $X \drar X^m$ is $(X,\Delta)$-non-positive.
Now, let $P$ be any prime divisor on $X$.
Then, by construction, we have $a_P(X,\Delta)=a_P(X',\Delta'+(1-c)E')$.
Notice that, by construction, $X' \drar X^m$ is $(X',\Delta'+(1-c)E')$-negative.
Therefore, if $P$ is contracted by $X' \drar X^m$, it follows that $a_P(X^c,\Delta^m) > a_P(X',\Delta'+(1-c)E')$.
Thus, $X \drar X^m$ is $(X,\Delta)$-negative.
This concludes the proof.
\end{proof}

\section{Minimal models with $\kappa(X,\Delta)=\dim(X)-1$}

In this section, we collect some definitions and facts about projective varieties of dimension $n$ and Kodaira dimension $n-1$.

\begin{remark} \label{remark Iitaka}
Let $(X,\Delta)$ be a minimal projective klt pair of dimension $n$ with $\kappa(X,\Delta)=n-1$ and rational coefficients.
In view of Theorem \ref{good mm} and \cite{Lai11}*{Proposition 2.4}, the Iitaka fibration of $(X,\Delta)$ is a morphism $f \colon X \rar Y$.
Furthermore, we have that $\K X. + \Delta \sim \subs \qq. f^*L$, where $L$ is an ample $\qq$-Cartier divisor.
Then, by the canonical bundle formula recalled in \S~\ref{section_cbf}, $f$ induces a generalized pair $(Y,B_Y,\bM.)$ such that $\K Y. + B_Y + \bM Y. \sim_\qq L$.
Furthermore, by \cite{FM00} (see also \cite{PS09}*{Theorem 8.1} and \cite{HMX14}), we may choose $0 \leq \Delta_Y \sim_\qq \bM Y.$ so that $(Y,B_Y+\Delta_Y)$ is klt, and the coefficients of $B_Y+\Delta_Y$ belong to a DCC set of rational numbers only depending on $n$ and $\mathrm{coeff}(\Delta)$.
In this setup, we claim that the discrepancies of $(Y,B_Y+\Delta_Y)$ are a lower bound for the generalized discrepancies of $(Y,B_Y,\bM.)$.
Indeed, let $\pi \colon Y' \rar Y$ be a log resolution of $(Y,B_Y)$ where $\bM.$ descends.
We write $\K Y.+B_{Y'}+\bM Y'.=\pi^*(K_Y+B_Y+\bM Y.)$.
As $\bM.$ descends on $Y$, the generalized discrepancies of $(Y,B_Y,\bM.)$ coincide with the discrepancies of the sub-pair $(Y',B_{Y'})$.
Then, to define $\Delta_Y$, by \cite{PS09}*{Theorem 8.1}, we choose a suitable effective divisor $0\leq \Delta_{Y'} \sim_{\qq} \bM Y'.$ and set $\Delta_Y \coloneqq \pi_*\Delta_{Y'}$.
Thus, the discrepancies of $(Y,B_Y+\Delta_Y)$ coincide with the discrepancies of $(Y',B_{Y'}+\Delta_{Y'})$, and the claim follows as $B_{Y'} \leq B_{Y'}+\Delta_{Y'}$.
\end{remark}

\begin{definition} \label{def volume}
Let $(X,\Delta)$ be a minimal projective klt pair of dimension $n$ with $\kappa(X,\Delta)=n-1$ and rational coefficients.
Let $Y$ and $L$ be as in Remark \ref{remark Iitaka}.
Then, we define $\vol \subs n-1.(X,\K X. + \Delta) \coloneqq \vol (Y,L)$.
\end{definition}

\begin{definition}
Fix positive integers $n,d$, a positive real number $v$, and a finite set $\mathcal{R} \subset \qq \cap [0,1]$.
We define $\mathfrak{D}(n,v,\Phi(\mathcal{R}))$ to be the set of minimal projective klt pairs $(X,\Delta)$ of dimension $n$ with $\kappa(X,\Delta)=n-1$, $\vol \subs n-1.(X,\K X.+\Delta)=v$, and $\coeff(\Delta) \subset \Phi(\mathcal R)$.
Then, we define $\mathfrak D (n,v,\Phi(\mathcal R), d) \subset \mathfrak D (n,v, \Phi(\mathcal{R}))$ to be the subset consisting of those $(X,\Delta) \in \mathfrak D (n,v,\Phi(\mathcal R))$ whose Iitaka fibration admits a multisection of degree $d$.
We define $\mathfrak{B}(n,v,\Phi(\mathcal R))$ to be the set of pairs $(Y,B_Y+\Delta_Y)$ arising as bases of the Iitaka fibration of the elements of $\mathfrak{D}(n,v,\Phi(\mathcal R))$ (i.e., $Y=\mathrm{Proj}(\oplus_{l \geq 0} \Gamma(X,\mathcal{O}_X(l(K_X+\Delta))))$).
Here, we choose $\Delta_Y$ as in Remark \ref{remark Iitaka}.
\end{definition}

\begin{proposition} \label{prop bdd base}
The set of pairs $\mathfrak{B}(n,v,\Phi(\mathcal{R}))$ is bounded.
\end{proposition}

\begin{proof}
Fix $(Y,B_Y+\Delta_Y) \in \mathfrak{B}(n,v,\Phi(\mathcal R))$.
Then, by definition of $\mathfrak{B}(n,v,\Phi (\mathcal R))$, we have that $\K Y. + B_Y + \Delta_Y$ is ample, and $\vol(Y,\K Y. + B_Y + \Delta_Y)=v$.
Since the coefficients of $B_Y+\Delta_Y$ belong to a DCC set that only depends on $n$ and $\mathcal R$, the claim follows from \cite{HMX18}*{Theorem 1.1}.
\end{proof}

\begin{corollary} \label{cor bound sings}
Fix positive integers $n$ and $d$, a positive real number $v$, and a finite set $\mathcal R \subset \qq \cap [0.1]$.
Then, there exists $\epsilon > 0$ only depending on $n$, $d$, $v$ and $\mathcal{R}$ such that, if $(X,\Delta) \in \mathfrak{D}(n,v,\Phi(\mathcal R))$, it is $\epsilon$-log canonical.
\end{corollary}

\begin{proof}
Fix $(X,\Delta) \in \mathfrak{D}(n,v,\Phi(\mathcal{R}))$, and let $f \colon X \rar Y$ be its Iitaka fibration.
Let $(Y,B_Y,\bM.)$ and $(Y,B_Y+\Delta_Y)$ be the generalized pair and pair induced on $Y$, respectively.
Since $\mathfrak{B}(n,v,\Phi(\mathcal R))$ is a bounded set of klt pairs with coefficients in a finite set, and since the generalized discrepancies of $(Y,B_Y,\bM.)$ are bounded below by the discrepancies of $(Y,B_Y+\Delta_Y)$, there exist $\epsilon>0$ only depending on $n$, $d$, $v$ and $\mathcal{R}$ such that $(Y,B_Y,\bM.)$ is generalized $\epsilon$-log canonical.
By assumption, the generic fiber is a smooth log Calabi--Yau curve whose boundary has coefficients in $\Phi(\mathcal{R})$.
Thus, by global ACC \cite{HMX14}, the horizontal part of $\Delta$ can only attain finitely many values.
Let $1-\delta$ denote the maximum of such values.
Notice that we have $0 < \delta \leq 1$, as $(X,\Delta)$ is klt.
Then, we may assume that $\epsilon < \delta$.
Thus, $(X,\Delta)$ is $\epsilon$-log canonical over the generic point of $Y$.
Indeed, as $f$ has relative dimension 1, the generic fiber of $f$ is smooth and in a neighborhood of it the singularities of $(X,\Delta)$ are controlled by the coefficients of $\Delta$.
In particular, over a non-empty open subset of $Y$, $(X,\Delta)$ is $\delta$-log canonical.
Therefore, in order to bound the singularities of $(X,\Delta)$, it suffices to bound the discrepancies of those geometric valuations over $X$ whose centers on $X$ are vertical over $Y$.
By \cite{Amb99}*{Proposition 3.4}, $(X,\Delta)$ is $\epsilon$-log canonical.
Notice that we can apply \cite{Amb99}*{Proposition 3.4} to the model $Y$ even though $\bM.$ does not descend on $Y$, since we are considering the generalized log discrepancies of $(Y,B_Y,\bM.)$.
\end{proof}

\begin{definition}
Fix positive integers $n$ and $d$, a positive real number $v$, and a finite set $\mathcal R \subset \qq \cap [0,1]$.
We define $\mathfrak{C}(n,v,\Phi(\mathcal R),d)$ to be the set of fibrations $(X,\Delta) \rar (Y,B_Y+\Delta_Y)$, where $(X,\Delta) \in \mathfrak{D}(n,v,\Phi(\mathcal{R}),d)$ and $(Y,B_Y+\Delta_Y)$ is the corresponding element in $\mathfrak{B}(n,v,\Phi(\mathcal R))$.
\end{definition}

\section{Iitaka fibration of Fano type}

In this section, we prove a stronger version of Theorem \ref{thm intro} under the additional assumption that the general fiber of the Iitaka fibration is $\pr 1.$.

\begin{theorem} \label{theorem fano case}
Let $\mathfrak F (n,v,\Phi(\mathcal{R}))$ be the subset of $\bigcup_{d \in \nn} \mathfrak{C}(n,v,\Phi(\mathcal{R}),d)$ consisting of those fibrations whose general fiber is $\pr 1.$.
Then, $\mathfrak F (n,v,\Phi(\mathcal{R}))$ is bounded.
\end{theorem}

\begin{proof}
By Proposition \ref{prop bdd base}, the set of pairs $(Y,B_Y+\Delta_Y)$ is bounded.
By Corollary \ref{cor bound sings}, $(X,\Delta)$ is $\epsilon$-log canonical for a fixed $\epsilon > 0$.
Therefore, by \cite{Bir18}*{Theorem 1.3}, the set of pairs $(X,\Delta)$ is bounded.
Let $(\mathcal{X},\mathcal{B}) \rar T$ be a bounding family.
We recall that the notions of boundedness introduced in \S~\ref{sec_bddness} only control the support of the boundary divisor $\Delta$ and that the divisor $\mathcal{B}$ is reduced.
Yet, if we want to retrieve the Iitaka fibration of the pairs $(X,\Delta)$, we need to put the appropriate coefficients to $\mathcal{X}$ to retrieve each $(X,\Delta)$, rather than $(X,\Supp(\Delta))$.
A priori, the coefficients of $\Delta$ belong to the countable set $\Phi(\mathcal{R})$.
If we argue that they actually belong to a finite subset of $\Phi(\mathcal{R})$, then, by considering finitely many copies of $(\mathcal{X},\mathcal{B}) \rar T$, we can assign to the components of $\mathcal{B}$ all the possible coefficients.
But this is indeed the case.
By global ACC \cite{HMX14}, the coefficients of the horizontal part of $\Delta$ belong to a finite subset of $\Phi(\mathcal{R})$.
Then, by Proposition~\ref{prop bdd base} (note that the possible coefficients appearing therein are finitely many by \cite{HMX18}*{Theorem 1.1}) and the algorithm to compute the boundary divisor in the canonical bundle formula, it readily follows that also the vertical components of $\Delta$ can attain finitely many coefficients.
Thus, up to replacing $(\mathcal{X},\mathcal{B}) \rar T$ with finitely many copies of itself and assigning coefficients to $\mathcal{B}$, we may assume that for every $(X,\Delta)$ there is a closed point $t \in T$ such that $(\mathcal{X}_t,\mathcal{B}_t) \cong (X,\Delta)$.
Furthermore, we may stratify $T$ so that $T$ is smooth and there exists a log resolution of $(\mathcal{X},\mathcal{B})$ that induces a log resolution fiberwise.
In particular, this guarantees that deformation invariance of plurigenera applies  \cite{HMX18b}*{Theorem 1.9.2}.
Thus, the relative ample model of $(\mathcal{X},\mathcal{B}) \rar T$ induces the ample model of each fiber; see, e.g., \cite{FS22}*{Theorem 4.5}.
Call the relative ample model $\mathcal{Y}$.
Thus, the claim of the theorem follows, besides the boundedness of the boundary divisors.
The latter is recovered by Corollary \ref{stratify cbf}.
\end{proof}

\section{Elliptic Iitaka fibration} \label{section elliptic}

Recall that $\mathfrak{D}(n,v,\Phi(\mathcal{R}))$ is the set of minimal projective klt pairs $(X,\Delta)$ of dimension $n$ with coefficients in $\Phi(\mathcal{R})$ and $\vol_{n-1}(X,\Delta)=v$.
In this section, we will only consider elements of $\mathfrak{D}(n,v,\Phi(\mathcal{R}))$ whose Iitaka fibration is an elliptic fibration.
Now, consider $(X,\Delta) \in \mathfrak{D}(n,v,\Phi(\mathcal{R}))$, and let $f \colon X \rar Y$ be its Iitaka fibration.
Then, we may consider the corresponding Jacobian fibration $J(X) \rar Y$.
This fibration is defined up to birational equivalence over $Y$, and the generic fiber is $\mathrm{Pic}^0(X_\eta)$, where $X_\eta$ denotes the generic fiber of $f$.
Indeed, in general $J(X)$ is a projective variety over $Y$ having $\mathrm{Pic}^0(X_\eta)$ as generic fiber of the morphism $J(X) \rar Y$.
Then, depending on the context, some additional assumptions, such as smoothness or minimality over $Y$, may be required.
Notice that neither smoothness nor relative minimality lead to a unique choice of representative of this birational class.
For our purposes, we will proceed as follows.
Assume that $J(X)$ is smooth and admits a contraction to $Y$.
Then, by \cite{HX13}*{Theorem 1.1}, $J(X)$ admits a relative good minimal model over $Y$.
Replace $J(X)$ with this model, and let $Y' \rar Y$ be the relative ample model over $Y$.
Then, we may assume that $J(X)$ is terminal, $\qq$-factorial, and that $\K J(X). \sim \subs \qq,Y'. 0$.
Let $(Y,B_Y,\bM.)$ be the generalized pair induced by $X$ on $Y$, and let $(Y',B_{Y'},\bM.)$ denote its trace on $Y'$.
Then, by Lemma \ref{lemma cbf jacobian}, $B_{Y'}$ is effective.
Notice that, as $(Y,B_Y+\Delta_Y)$ is bounded and $\Delta_Y$ can be chosen generically with bounded denominator \cite{PS09}*{Theorem 8.1}, it follows that the coefficients of $B_{Y'}$ belong to a finite set only depending on $\mathfrak{D}(n,v,\Phi(\mathcal R))$.
Thus, it follows by \cite{MST} that $(Y',B_{Y'}+\Delta_{Y'})$ is bounded with coefficients in a finite set.
Then, by Lemma \ref{lemma cbf jacobian}, we have that $J(X)$ induces a generalized pair $(Y',C_{Y'},\bM.)$, where $0 \leq C_{Y'} \leq B_{Y'}$.
By \cite{FM00}, the coefficients of $C_{Y'}$ belong to a DCC set whose only accumulation point is 1.
As they are bounded away from 1 by the coefficients of $B_{Y'}$, it follows that $\coeff(C_{Y'})$ varies in a finite set only depending on $\mathfrak{D}(n,v,\Phi(\mathcal R))$.

\begin{definition} \label{def fibration J}
Fix positive integers $n$ and $d$, a positive real number $v$, and a finite set of rational numbers $\mathcal{R} \subset \qq \cap [0,1]$.
We define $\mathfrak{J}(n,v,\Phi(\mathcal R))$ to be the set of Jacobian fibrations constructed above.
That is, $J(X) \in \mathfrak{J}(n,v,\Phi(\mathcal R))$ is a terminal, $\qq$-factorial model of the Jacobian fibration of $(X,\Delta) \in \mathfrak{D}(n,v,\Phi(\mathcal R))$.
Furthermore, $J(X)$ is relatively minimal over $Y$, the base of the Iitaka fibration of $X$, with relative ample model $Y' \rar Y$.
Then, we define $\mathfrak{P}(n,v,\Phi(\mathcal R))$ to be the set of fibrations $J(X) \rar (Y',B_{Y'}+\Delta_{Y'})$, where $J(X) \in \mathfrak{J}(n,v,\Phi(\mathcal R))$ and $(Y',B_{Y'}+\Delta_{Y'})$ is the trace on $Y'$ of the corresponding element in $\mathfrak{B}(n,v,\Phi(\mathcal R))$.
\end{definition}

\begin{theorem} \label{main thm}
Fix positive integers $n$ and $d$, a positive real number $v$, and a finite set of rational numbers $\mathcal{R} \subset \qq \cap [0,1]$.
The subset of $\mathfrak{C}(n,v,\Phi(\mathcal R),d)$ consisting of elliptic fibrations is bounded in codimension 1.
Let $(\mathcal{X},\mathcal{D})$, $(\mathcal{Y},\mathcal{D})$, and $T$ be varieties bounding these fibrations in codimension 1 as in Definition~\ref{def_bdd_fibrations}.
Then, $(\mathcal{Y},\mathcal{D}) \rar T$ bounds (i.e., not just birationally in codimension 1) the set of bases of the fibrations of interest.
\end{theorem}

\begin{remark}
Notice that, if $n=2$, it follows that the family of elliptic fibrations in Theorem \ref{main thm} is actually bounded.
Indeed, two normal projective surfaces that are isomorphic in codimension 1 are actually isomorphic.
\end{remark}

\begin{proposition} \label{bound cycles}
Fix positive integers $n$ and $d$, a positive real number $v$, and a finite set of rational numbers $\mathcal{R} \subset \qq \cap [0,1]$.
Assume that $n \geq 3$.
Assume that Theorem \ref{main thm} holds for $\mathfrak{C}(m,w,\Phi(\mathcal{S}),e)$, where $m \leq n-1$, $w$, $e$ and the finite set $\mathcal S \subset \qq \cap [0,1]$ are arbitrary.
Then, there exists a positive integer $C$, only depending on $\mathfrak{C}(n,v,\Phi(\mathcal R),d)$, such that, for every $(X,\Delta) \rar (Y,B_Y+\Delta_Y) \in \mathfrak{C}(n,v,\Phi(\mathcal R),d)$, the linear series $|C(\K Y. + B_Y + \Delta_Y)|$ is very ample and there exists an element $D \in |C(\K Y. + B_Y + \Delta_Y)|$ such that $Y_{\mathrm{sing}}\subset \Supp(D)$ and $X \rar Y$ is smooth over $Y \setminus \Supp(D)$, except possibly at finitely many isolated points.
\end{proposition}

\begin{proof}
By Proposition \ref{prop bdd base}, $\mathfrak{B}(n,v,\Phi(\mathcal R))$ is bounded.
Fix $f \colon (X,\Delta) \rar (Y,B_Y+\Delta_Y) \in \mathfrak{C}(n,d,\Phi(\mathcal R),v)$.
By Proposition \ref{prop bdd base}, there is a positive integer $C'$, only depending on $n$ and $v$, such that $C'(\K Y. + B_Y + \Delta_Y)$ is a very ample Cartier divisor.
Pick a general $H \in |C'(\K Y. + B_Y + \Delta_Y)|$, and define $X_H \coloneqq f^* H$, and let $(X_H,\Delta_{X_H})$ be the pair obtained by adjunction.
As $H$ is general, it follows that $\coeff(\Delta_{X_H}) \subset \Phi(\mathcal R)$.
By adjunction, we define $(H,B_H+\Delta_H)$ and $(H,B_H,\bN.)$, where $\bN. \coloneqq \bM.|_H$.
Notice that $\Delta_H \sim_\qq \bN H.$.
Then, arguing as in \cite{Flo14}*{Lemma 3.1} (notice that the smoothness assumptions in the reference are not needed, and it suffices that $H$ is chosen generically in a free linear series, as the computations can be done on a log resolution of $Y$ and then pushed forward to $Y$ and $H$), it follows that $(H,B_H,\bN.)$ is the generalized pair induced by $X_H \rar H$ via the canonical bundle formula.
Furthermore, $(X_H,\Delta_{X_H}) \rar (H,B_H+\Delta_H) \in \mathfrak{C}(n-1,w,\Phi(\mathcal R),d)$, where $w$ only depends on $\mathfrak{C}(n,v,\Phi(\mathcal R),d)$.
Notice that in this argument we are using the assumptions on the singularities of $X$ and $H$ to control the behavior of the different \cite{Kol13}*{Proposition 4.5}.
\newline
Let $Y_\mathrm{sing}$ and $H_\mathrm{sing}$ denote the singular loci of $Y$ and $H$, respectively.
Then, by Bertini's theorem, we have $H_\mathrm{sing}=Y_\mathrm{sing} \cap H$.
Define $Y_\mathrm{sm} \coloneqq Y \setminus Y_\mathrm{sing}$ and $H_\mathrm{sm} \coloneqq H \setminus H_\mathrm{sing}$.
Let $\Sigma \subs Y_\mathrm{sm}.$ denote the minimal closed subset of $Y_{\mathrm{sm}}$ such that $X \rar Y$ is a smooth fibration with smooth base over $Y_\mathrm{sm} \setminus \Sigma \subs Y_\mathrm{sm}.$.
Let $\Sigma_Y$ denote the closure of $\Sigma \subs Y_\mathrm{sm}.$ in $Y$.
Since $H$ is general and $\Sigma \subs Y.$ is independent of the choice of $H$, we may assume that $H$ meets properly every irreducible component of $\Sigma_Y$.
Then, we let $\Sigma \subs H_\mathrm{sm}.$ denote the minimal closed subset of $H_{\mathrm{sm}}$ such that $X_H \rar H$ is a smooth fibration with smooth base over $H_\mathrm{sm} \setminus \Sigma \subs H_\mathrm{sm}.$, and we let $\Sigma_H$ denote its closure in $H$.
By construction, we have $\Sigma \subs H_\mathrm{sm}.=\Sigma \subs Y_\mathrm{sm}. \cap H_\mathrm{sm}$.
This is clear over the locus where the fibers are not 1-dimensional.
Then, over the locus in $Y_{\mathrm{sm}}$ where the fibers are 1-dimensional, the morphism is flat by \cite{Sta}*{Tag 00R4}.
Thus, over this locus, by \cite{Sta}*{Tag 01V9}, the morphism is smooth if and only if so is the restriction to $H_{\mathrm{sm}}$, as the fibers over closed points agree scheme theoretically.
Hence, we have $\Sigma_H = \Sigma_Y \cap H$.
Notice that, by Lemma \ref{lemma same smooth fibers 1} and Lemma \ref{lemma same smooth fibers 2}, these loci are well defined, even though we may need to replace $X$ and $X_H$ with models that are isomorphic in codimension 1.
\newline
By the inductive hypothesis, Theorem \ref{main thm} applies to $\mathfrak{C}(n-1,w,\Phi(\mathcal R),d)$.
Let $(\mathcal{X}_\mathcal{H},\mathcal{B}_\mathcal{H}) \rar (\mathcal{H},\mathcal{D}) \rar T$ be the corresponding family of fibrations.
Up to a stratification, we may assume that $T$ is smooth.
Then, we may assume that the singularities of the fibers of $\mathcal{H} \rar T$ are induced by the singularities of $\mathcal{H}$ itself.
That is, we may assume that $\mathcal{H} \subs t,\mathrm{sing}.= (\mathcal{H}_\mathrm{sing})_t$.
Let $\Sigma \subs \mathcal{H}_\mathrm{sm}.$ be the minimal closed subset of $\mathcal{H}_\mathrm{sm}$ such that $\mathcal{X}_\mathcal{H} \rar \mathcal{H}$ is a smooth fibration over $\mathcal{H}_\mathrm{sm} \setminus \Sigma \subs \mathcal{H}_\mathrm{sm}.$.
By a stratification, we may assume that $(\Sigma \subs \mathcal{H}_\mathrm{sm}.)_t=\Sigma \subs \mathcal{H}_{t,\mathrm{sm}}.$.
Similalry, we have $(\Sigma \subs \mathcal H .)_t=\Sigma \subs \mathcal{H}_t.$, where $\Sigma \subs \mathcal H .$ and $\Sigma \subs \mathcal{H}_t.$ denote the closures of $\Sigma \subs \mathcal{H}_\mathrm{sm}.$ in $\mathcal{H}$ and and of $\Sigma \subs \mathcal{H}_{t,\mathrm{sm}}.$ in $\mathcal{H}_t$, respectively.
\newline
By \cite{Kol13}*{Claim 4.38.1}, up to a stratification and a base change, we may assume that the restrictions of the irreducible components of $\mathcal{B}_\mathcal{H}$ and $\mathcal{D}$ to the fibers of the corresponding morphisms to $T$ are irreducible.
By Corollary \ref{cor bound sings}, the coefficients of $\Delta_{X_H}$ can attain finitely many values.
Thus, as the coefficients of $\Delta_{X_H}$ and $B_H+\Delta_H$ take finitely many values, we may assume that $\mathcal{B}_\mathcal{H}$ restricts to $\Delta_{X_H}$ fiberwise.
Similarly, we may assume that $\mathcal{D}$ restricts to $B_H + \Delta_H$.
Then, by \cite{HX15}*{Proposition 2.4} and Corollary \ref{cor bound sings}, up to a stratification, we may assume that $\K \mathcal{X}_\mathcal{H}. + \mathcal{B}_\mathcal{H}$ and $\K \mathcal{H}. + \mathcal{D}$ are $\qq$-Cartier.
\newline
By generic flatness, up to a stratification of $T$, we may assume that $\mathcal{D}$ and every irreducible component of $\Sigma_\mathcal{H}$ are flat over $T$.
Now, write $T = \sqcup T_i$, where each $T_i$ is irreducible, and let $\mathcal{X} \subs \mathcal{H},i.$ and $\mathcal{H}_i$ be the corresponding irreducible components of $\mathcal{X}_\mathcal{H}$ and $\mathcal{H}$, respectively.
Fix an irreducible component $\Xi_i$ of $\Sigma \subs \mathcal{H}_i.$ of dimension $\dim(T_i)+k$.
Recall that $\dim (X)=n$.
Therefore, we have $\dim(H)=n-2$.
Then, for every $t \in T_i$, we have
\begin{equation} \label{equation_intersections}
(\K \mathcal{H},t.+ \mathcal{D}_t) \sups k. \cdot \Xi \subs i,t. = (\K \mathcal{H}.+\mathcal{D})^{k} \cdot \Xi \subs i. \cdot \mathcal{H}_t,
\end{equation}
and this expression is independent of $t \in T_i$.
Thus, the intersection products between $\K \mathcal H ,t. + \mathcal{D}_t$ and the irreducible components of $\Sigma \subs \mathcal{H},t.$ are locally constant.
In particular, they are bounded.
\newline
Now, recall that for every pair $(H,B_H+\Delta_H)$ appearing as a fiber, we have
\[
\K \mathcal{H},t.+ \mathcal{D}_t = \K H. + B_H + \Delta_H.
\]
By construction, there is $C' \in \nn$ only depending on $n$, $v$ and $\mathcal R$ so that $H \sim C'(\K Y. + B_Y + \Delta_Y)$.
For brevity, set $A_Y \coloneqq C'(\K Y. + B_Y + \Delta_Y)$.
Notice that $A_Y$ is a very ample polarization.
Now, fix an irreducible component $\Omega$ of $\Sigma_Y$ of dimension $l \geq 1$, and set $\Omega_H \coloneqq \Omega \cap H$.
Then, by the construction of $\Sigma_Y$ and $\Sigma_H$, we have
\[
\Omega \cdot A^{l} = \Omega \cdot H \cdot A^{l-1} = \Omega_H \cdot A_H^{l-1} = (C')^{l-1} \Omega_H \cdot (\K H. + B_H + \Delta_H)^{l-1},
\]
and the right hand side is bounded, by \eqref{equation_intersections} and the fact that $\K H. + B_H + \Delta_H \sim_\qq \K \mathcal{H},t. + \mathcal{D}_t$ for some $t \in T$.
Thus, the positive dimensional irreducible components of $\Sigma_Y$ have bounded degree and they are bounded in number.
Let $M \in \nn$ be a positive integer that bounds both of these quantities.
\newline
Now, recall that the degree with respect to $A$ of a subvariety of $X$ is the same as the degree of the same variety with respect to the hyperplane class of $\pr .(H^0(X,\O X.(A)))$.
Thus, every irreducible component $\Omega$ of $\Sigma_Y$ of positive dimension is a subvariety of $\pr .(H^0(X,\O X.(A)))$ of degree at most $M$.
Thus, set-theoretically, $\Omega$ is the intersection of hypersurfaces of degree at most $M$.
Choosing one of these hypersurfaces generically so that it does not contain $X$, it follows there is $D_\Omega \in |MA|$ such that $\Omega \subset \Supp(D_\Omega)$.
Thus, we may find an element $D \in |M^2A|$ whose support contains the positive dimensional part of $\Sigma_Y$.
Notice that the degree of the singular locus of $Y$ is bounded, as $Y$ comes in a bounded family.
Thus, we may also assume that $\Supp(D)$ contains the singular locus of $Y$.
Thus, the claim follows.
\end{proof}

\begin{remark} \label{remark support}
In the proof of Proposition \ref{bound cycles}, we may also assume that $\Supp(B_Y+\Delta_Y) \subset \Supp(D)$.
Indeed, as these divisors are bounded as the cycle $\Sigma_Y$ is, we may find a bounded multiple of $\K Y. +B_Y + \Delta_Y$ such that some element of the corresponding linear series vanishes along these divisors.
\end{remark}

\begin{remark} \label{rmk D' bdd}
Fix $X \in \mathfrak{C}(n,v,\Phi(\mathcal R),d)$ and the associated $(Y,B_Y+\Delta_Y+D)$, where $D$ is as in Proposition \ref{bound cycles}.
Also, fix $J(X)$ as discussed above.
Consider $(Y',B_{Y'}+\Delta_{Y'})$ as in Definition \ref{def fibration J}, and let $D'$ be the pull-back of $D$ to $D'$.
Since $(Y',B_{Y'}+\Delta_{Y'})$ is bounded and $0 \leq D' \sim_\qq C (\K Y'. + B \subs Y'. + \Delta \subs Y'.)$ where $C$ only depends on $n$, $v$, $d$, and $\mathcal{R}$, it follows that $(Y',B_{Y'}+\Delta_{Y'}+D')$ is bounded.
Then, we can regard $J(X) \rar (Y',B_{Y'}+\Delta_{Y'}+D')$ as an element of $\mathfrak{P}(n,v,\Phi(\mathcal R))$ together with the datum of $D'$.
Notice that $D'$ depends on $\mathfrak{C}(n,v,\Phi(\mathcal R),d)$, and not just on $\mathfrak{C}(n,v,\Phi(\mathcal R))$.
\end{remark}

\begin{definition}
We define $\mathfrak{P}(n,v,\Phi(\mathcal R),d)$ to be the set of fibrations $J(X) \rar (Y',B_{Y'}+\Delta_{Y'}+D')$, where $J(X) \rar (Y',B_{Y'}+\Delta_{Y'}) \in \mathfrak{P}(n,d,\Phi(\mathcal R))$ and $D'$ is the pull-back of $D$, which is as in Proposition \ref{bound cycles}.
\end{definition}

\begin{theorem} \label{bdd jacobian}
Fix positive integers $n,d$, a positive real number $v$, and a finite set $\mathcal{R} \subset \qq \cap [0,1]$.
Assume that Theorem \ref{main thm} holds true in dimension strictly less than $n$.
Then, the set of fibrations $\mathfrak{P}(n,v,\Phi(\mathcal{R}),d)$ is bounded in codimension 1.
Furthermore, we may choose the birationally bounding family so that the bases of the firbations are actually bounded.
\end{theorem}

\begin{remark}
Notice that the inductive assumption on Theorem \ref{main thm} is only needed for the existence of the divisor $D'$.
\end{remark}

\begin{proof}
Fix a fibration $j \colon J(X) \rar (Y',B_{Y'}+\Delta_{Y'}+D') \in \mathfrak{P}(n,v,\Phi(\mathcal R),d)$.
Then, by construction, we have $\K J(X). \sim_\qq j^* (\K Y'. + C_{Y'} + \Delta_{Y'})$, where $0 \leq C_{Y'} \leq B_{Y'}$.
By Remark \ref{rmk D' bdd}, the pairs of the form $(Y',B_{Y'}+\Delta_{Y'}+D')$ are bounded.
Since $(Y',B_{Y'}+\Delta_{Y'})$ is bounded and klt with coefficients in a finite set, we may find $0 < c \ll 1$ only depending on $\mathfrak{C}(n,v,\Phi(\mathcal R),d)$ so that the pairs of the form $(Y',B_{Y'}+\Delta_{Y'}+cD')$ are klt.
This can be achieved by taking a bounding family for the pairs $(Y',B_{Y'}+\Delta_{Y'}+D')$, where we prescribe the coefficients of the bounding divisors to restrict to $B \subs Y'. + \Delta \subs Y'.$.
Notice that, by the construction of $D$ in Proposition \ref{bound cycles}, $\K Y'. + B_{Y'} + \Delta_{Y'} + cD'$ is semi-ample and big with fixed volume.
Now, we divide the proof into several steps.
We follow the proof of \cite{dCS17}*{Theorem 1.1}.
\newline
{\bf Step 1:} {\it In this step we reduce to the case when $J(X)$ maps to a $\qq$-factorialization $Y \sups \qq.$ of $Y'$.
}
\newline
Since $(Y',B_{Y'}+\Delta_{Y'}+cD')$ is klt, $Y'$ admits a small $\qq$-factorialization $(Y^\qq,B_Y^\qq+\Delta_Y^\qq+cD^\qq) \rar (Y',B_{Y'}+\Delta_{Y'}+cD')$.
By \cite{MST}, also $(Y^\qq,B_Y^\qq+\Delta_Y^\qq+cD^\qq)$ belongs to a bounded family which only depends on $\mathfrak{C}(n,v,\Phi(\mathcal R),d)$.
Up to replacing $J(X)$ with a model that is isomorphic in codimension 1, by Proposition \ref{small q-fact}, we may assume that $J(X) \rar Y'$ factors through $Y^\qq$.
\newline
{\bf Step 2:} {\it In this step we reduce to the case when the rational section satisfies certain positivity assumptions.}
\newline
Now, denote by $\hat Y$ the closure of the rational section of $j \colon J(X) \rar Y'$.
Then, $\hat Y$ is relatively big over $Y^\qq$.
Also, for $0 < \gamma \ll 1$, $(J(X),\gamma \hat Y)$ is klt.
Thus, by \cite{BCHM}, any $(\K J(X). + \gamma \hat Y)$-MMP over $Y^\qq$ with scaling of an ample divisor terminates.
Let $(\tilde J (X),\gamma \tilde Y)$ be the resulting model.
Denote by $\tilde j \colon \tilde{J}(X) \rar Y^\qq$ the resulting morphism.
Notice that $\K J(X). + \gamma \hat Y \sim \subs \qq,Y^\qq. \gamma \hat Y$.
Thus, this MMP is independent of $\gamma$, and $\tilde Y$ is relatively big and semi-ample over $Y^\qq$.
Furthermore, since $\hat Y$ is irreducible and dominates $Y^\qq$, every step of the above MMP has to be a $(\K J(X). + \gamma \hat Y)$-flip.
Thus, $\tilde{J}(X)$ is isomorphic to $J(X)$ in codimension 1.
Moreover, as $\K J(X). \sim \subs \qq,Y^\qq. 0$, the terminality of $J(X)$ implies that of $\tilde{J}(X)$.
Thus, up to relabelling, we may assume that $(J(X),\hat Y)=(\tilde J (X),\tilde Y)$.
\newline
{\bf Step 3:} {\it In this step we show that $(J(X),\hat Y)$ is a plt pair.
This implies that $\hat Y$ is a normal variety admitting a structure of klt pair.}
\newline
Normality of $\hat Y$ and the exitence of a structure of klt pair on $\hat Y$ will follow from the plt-ness of $(J(X),\hat Y)$ by \cite{KM98}*{Proposition 5.51} and inversion of adjunction, see \cite{Kaw07}.
To show that $(J(X),\hat Y)$ is plt, it suffices to show that $(J(X), \hat Y)$ is log canonical and that $\hat Y$ is its only log canonical center.
Let $\phi \colon \hat Y ^\nu \rar \hat Y$ be the normalization of $\hat Y$, and let $\mathrm{Diff}(0)$ be the different defined by
\[
\K \hat Y ^\nu. + \mathrm{Diff}(0) \coloneqq \phi^*((\K J(X). + \hat Y)| \subs \hat Y.).
\]
By construction, $\K \hat Y ^\nu. + \mathrm{Diff}(0)$ is nef and big over $Y^\qq$.
By \cite{dCS17}*{Lemma 5.1}, $\mathrm{Diff}(0)$ is exceptional over $Y^\qq$.
Thus, we have $(\hat j \circ \phi)_*(\K \hat Y ^\nu. + \mathrm{Diff}(0))=\K Y^\qq.$, where we set $\hat j \colon J(X) \rar Y^\qq$.
Since $Y^\qq$ is $\qq$-factorial, the negativity lemma \cite{KM98}*{Lemma 3.39} implies that
\begin{align}
\label{eqtn_klt}
\K \hat Y ^\nu. + \mathrm{Diff}(0) = (\hat j \circ \phi)^* \K Y^\qq. - F, 
\end{align}
where $F \geq 0$ is $(\hat j \circ \phi)$-excetpional.
As $(Y^\qq,B_Y^\qq+\Delta_Y^\qq+cD^\qq)$ is klt, then so is $(Y^\qq,0)$.
Therefore, it follows from \eqref{eqtn_klt} that $(\hat Y ^\nu,\mathrm{Diff}(0))$ is klt.
Inversion of adjunction implies that $\hat Y$ is the only log canonical center of $(J(X), \hat Y)$.
In particular, $(J(X), \hat Y)$ is plt and the other conclusions follow as indicated above.
\newline
{\bf Step 4:} {\it In this step 
we show that there exists an effective divisor $\hat G$ on $J(X)$ such that the pair $(J(X), \frac 1 2 \hat Y + \frac 1 2 \hat G)$ is $\frac 1 2$-log canonical, and $K_{J(X)}+\frac 1 2 \hat Y + \frac 1 2 \hat G$ is big.}
\newline
Let $H'$ be a very ample polarization of bounded degree on $Y'$.
Notice that its existence is guaranteed by the boundedness of the pairs $(Y',B_{Y'}+\Delta_{Y'}+cD')$.
We may assume that $H' \pm (\K Y'. + B_{Y'} + \Delta_{Y'})$ and $H' \pm (\K Y'. + C_{Y'} + \Delta_{Y'})$ are ample.
Since $D'$ is $\qq$-linearly equivalent to $\K Y. + C \subs Y'. + \Delta \subs Y'.$ up to a bounded multiple, we may also assume that $H' \pm D'$ is ample.
Furthermore, we may assume that $H^0(Y',\O Y'.(H'-\Supp(B_{Y'}+\Delta_{Y'}))\neq 0$.
Let $\hat{G}$ be a general member of $|(2n+3)  j ^*H'|$.
Then, $(J(X),\hat{Y}+\hat G)$ is log canonical.
On the other hand, $J(X)$ is terminal, and the discrepancies of valuations are linear functions of the boundary divisor of a pair.
Hence, it follows that $(J(X), \frac{1}{2} \hat Y + \frac{1}{2} \hat G)$ is $\frac{1}{2}$-log canonical.
Since $\K J(X).+j^*H'$ is the pull-back of an ample divisor on $Y'$, $\hat Y$ is effective and relatively big over $Y'$, it follows that $\K J(X). + \frac{1}{2}\hat Y + \frac{1}{2} \hat G$ is big.
Since we have
\[
\K J(X). + \frac{1}{2}\hat Y + \frac{1}{2} \hat G \sim_\qq \frac{1}{2} (\K J(X) . +j^*H') + \frac{1}{2} (\K J(X). + \hat Y + (2n+2)j^*H'),
\]
and $\K J(X). + j^*H'$ is nef, it suffices to show that $\K J(X). + \hat Y + (2n+2)j^* H'$ is nef to conclude that the same holds for $\K J(X). + \frac{1}{2}\hat Y + \frac{1}{2} \hat G$.
The nefness of $\K J(X). + \hat Y + (2n+2)j^* H'$ follows by the boundedness of the negative extremal rays \cite{Fuj14}*{Theorem 1.19}.
Indeed, let $R$ be a $(\K J(X). + \hat Y)$-negative extremal ray.
There exists a rational curve $C$ spanning $R$ such that $-2n \leq (\K J(X). + \hat Y)\cdot C < 0$.
Since $\K J(X). + \hat Y$ is nef relatively to $Y'$, then $j (C)$ is a curve.
In particular, we have $(2n+2)j^*H' \cdot C \geq (2n+2)H' \cdot j (C) \geq 2n+2$.
So, it follows that $\K J(X). + \hat Y + (2n+2)j^* H'$ is non-negative on every $(\K J(X). + \hat Y)$-negative extremal ray.
Thus, $\K J(X). + \hat Y + (2n+2)j^* H$ is nef.
In particular, we have that $\K J(X). + \frac{1}{2}\hat Y + \frac{1}{2} \hat G$ is nef and big.
\newline
{\bf Step 5:} {\it In this step we show that there exist positive constants $C_1$ and $C_2$, only depending on $\mathfrak{C}(n,v,\Phi(\mathcal R),d)$, such that $C_1 \leq (\K J(X). + \frac{1}{2}\hat Y + \frac{1}{2} \hat G)^n \leq C_2$.}
\newline
The existence of $C_1$ follows from \cite{HMX14}*{Theorem 1.3}.
Thus, we are left to show the existence of $C_2$.

Up to a rescaling factor only depending on $n$, we need to consider the following quantity:
\begin{equation}\label{eq_summands}
    2\K J(X). \cdot (2 \K J(X).+\hat{Y}+\hat{G})^{n-1}+\hat{Y} \cdot (2 \K J(X).+\hat{Y}+\hat{G})^{n-1}+\hat{G} \cdot (2 \K J(X).+\hat{Y}+\hat{G})^{n-1}.
\end{equation}
By the choice of $H'$ in Step 4, $\hat{G}-2\K J(X).$ is semi-ample.
Thus, as $2 \K J(X).+\hat{Y}+\hat{G}$ is nef, in order to bound $2\K J(X). \cdot (2 \K J(X).+\hat{Y}+\hat{G})^{n-1}$, it suffices to bound $\hat G \cdot (2 \K J(X).+\hat{Y}+\hat{G})^{n-1}$.
Thus, we are left with finding an upper bound for the second and third summands in \eqref{eq_summands}.

By \eqref{eqtn_klt} and adjunction, we have $(\K J(X).+\hat{Y})|_{\hat{Y}} = (\hat{j}|_{\hat{Y}})^* \K Y^\qq. - F$, where $F \geq 0$ is ${\hat{j}}|_{\hat{Y}}$-exceptional.
Thus, we get
\begin{align}\label{eq_second_summand}
    \begin{split}
        \hat{Y} \cdot (2 \K J(X).+\hat{Y}+\hat{G})^{n-1}&=((\K J(X).+\hat{Y})|_{\hat{Y}}+\K J(X).|_{\hat{Y}}+\hat{G}|_{\hat{Y}})^{n-1}\\
        &=\vol(\hat{Y},(\K J(X).+\hat{Y})|_{\hat{Y}}+\K J(X).|_{\hat{Y}}+\hat{G}|_{\hat{Y}})\\
        &\leq \vol(\hat{Y},(\hat{j}|_{\hat{Y}})^* \K Y^\qq.+\K J(X).|_{\hat{Y}}+\hat{G}|_{\hat{Y}})\\
        &\leq \vol(Y',(2n+5)H')\\
        &=(2n+5)^{n-1}(H')^{n-1},
    \end{split}
\end{align}
where the last quantity is bounded, as $n$ is fixed and $H'$ is the fixed very ample polarization of bounded degree $Y'$.
In \eqref{eq_second_summand}, the second inequality follows from \eqref{eqtn_klt}, while the third inequality follows from the projection formula and the fact that $H' - (\K Y'. + B_{Y'} + \Delta_{Y'})$ is ample.

Thus, we are left with bounding the summand $\hat{G} \cdot (2 \K J(X).+\hat{Y}+\hat{G})^{n-1}$.
For an integer $1 \leq k \leq n-1$, we have
\begin{align}\label{ineq_induction}
    \begin{split}
        \hat{G}^k \cdot (2 \K J(X).+\hat{Y}+\hat{G})^{n-k}&=\hat{G}^k \cdot (2 \K J(X).+\hat{Y}+\hat{G}) \cdot (2 \K J(X).+\hat{Y}+\hat{G})^{n-k-1}\\
        &\leq \hat{G}^k \cdot (\hat{Y}+2\hat{G}) \cdot (2 \K J(X).+\hat{Y}+\hat{G})^{n-k-1}\\
        &=2\hat{G}^{k+1} \cdot (2 \K J(X).+\hat{Y}+\hat{G})^{n-k-1}
        +\hat{G}^{k} \cdot \hat{Y} \cdot (2 \K J(X).+\hat{Y}+\hat{G})^{n-k-1},
    \end{split}
\end{align}
where the inequality follows from the fact that, by construction, $\hat{G}-2\K J(X).$ is semi-ample.
Since $\hat{G}^{n}=0$, by iteratively applying \eqref{ineq_induction} $n-1$ times, bounding the intersection number $\hat{G} \cdot (2 \K J(X).+\hat{Y}+\hat{G})^{n-1}$ reduces to bounding $\hat{G}^{k} \cdot \hat{Y} \cdot (2 \K J(X).+\hat{Y}+\hat{G})^{n-k-1}$ for $1 \leq k \leq n-1$.
This can be achieved by \eqref{eqtn_klt} and the fact that $H' - (\K Y'. + B_{Y'} + \Delta_{Y'})$ is ample as in the treatment of \eqref{eq_second_summand}.
Thus, this concludes the step.
\newline
{\bf Step 6:} {\it In this step we show that the birational representative of the Jacobian fibration $J(X)$ chosen at the end of Step 2 is bounded.}
\newline
As showed in the previous steps, $(J(X), \frac{1}{2}\hat Y + \frac{1}{2} \hat G)$ is $\frac{1}{2}$-log canonical and its coefficients belong to the set $\lbrace \frac{1}{2} \rbrace$.
Thus, by \cite{Fil19}*{Theorem 1.3}, $\vol (\K J(X). + \frac{1}{2} \hat Y + \frac{1}{2} \hat G )$ belongs to a discrete set only depending on $\mathfrak{C}(n,v,\Phi(\mathcal R),d)$.
By Step 5, this volume is also bounded from above and below.
Thus, we conclude that $\vol (\K J(X). + \frac{1}{2} \hat Y + \frac{1}{2} \hat G )$ attains only finitely many values, only depending on $\mathfrak{C}(n,v,\Phi(\mathcal R),d)$.
Then, by \cite{MST}*{Theorem 6}, the set of pairs $(J(X), \frac{1}{2}\hat Y + \frac{1}{2} \hat G)$ is bounded.
In particular, the varieties $J(X)$ are bounded.
Notice that the boundedness of this specific model guarantees the boundedness in codimension 1 of any other model with the properties required at the beginning of \S~\ref{section elliptic}, since the distinguished model chosen at the end of Step 2 is isomorphic in codimension 1 to any arbitrary model taken as input at the beginning of this proof.
\newline
{\bf Step 7:} \emph{In this step we conclude the proof.}
\newline
Since the boundary divisor $\hat G$ is bounded, the intersection between a fixed very ample polarization and $\hat G$ is bounded.
Now, we consider $j^*D'$.
By the choice of $H'$ in Step 4, $H'-D'$ is ample.
Thus, we have that the intersection between any polarization on $J(X)$ and $\Supp(j^*D')$ is bounded above by the intersection with $\hat G$.
Similarly, as $H^0(Y',\O Y'.(H'-\Supp(B_{Y'}+\Delta_{Y'}))\neq 0$, the divisor $H^\qq-\Supp(B_{Y^\qq}+\Delta_{Y^\qq})$ is linearly equivalent to an effective divisor.
Thus, we can bound $\Supp(\hat j^*(B_{Y^\qq}+\Delta_{Y^\qq}))$.
Hence, the pairs $(J(X),\hat Y + \hat G + \hat j^*(B_{Y^\qq}+\Delta_{Y^\qq}+D^\qq))$ are bounded, where we choose the birational representative of $J(X)$ according to the previous step.
Choose a bounding family $(\mathcal{J},\mathcal{E}) \rar U$, where the reduced divisor $\mathcal{E}$ bounds the support of all the above divisors.
That is, if the closed point $u \in U$ corresponds to $(J(X),\hat Y + \hat G + \hat j^*(B_{Y^\qq}+\Delta_{Y^\qq}+D^\qq))$, we have $(J(X),\Supp(\hat Y + \hat G + \hat j^*(B_{Y^\qq}+\Delta_{Y^\qq}+D^\qq))) \cong (\mathcal{J}_u,\mathcal{E}_u)$.
\newline
By Step 4, $\K J(X). + \frac{1}{2} \hat G$ is $\qq$-linearly equivalent to the pull-back of an ample divisor on $Y'$.
The pairs of the form $(J(X),\frac{1}{2}\hat G)$ are bounded and klt.
Thus, up to stratifying $U$ so that a log resolution of $(\mathcal{J},\mathcal{E})$ induces a log resolution of the fibers of $\mathcal{J} \rar U$, we may apply the conclusions of \cite{HMX18}*{Corollary 1.4}.
Furthermore, up to a base change, we may assume that the restrictions of the irreducible components of $\mathcal{E}$ to the fibers are irreducible \cite{Kol13}*{Claim 4.38.1}.
Thus, we can define the divisor $\mathcal{G}$, which is supported on $\mathcal{E}$ and restricts to the divisor $\hat G$ on the fibers of the family.
Up to a stratification of the family \cite{HX15}*{Proposition 2.4}, $\K \mathcal J. + \frac{1}{2} \mathcal{G}$ is $\qq$-Cartier.
Then, the relative ample model of $\K \mathcal J. + \frac{1}{2} \mathcal{G}$ induces the ample model of each fiber.
Thus, the family $\mathcal{J} \rar U$ factors as $\mathcal{J} \rar \mathcal{Y}' \rar U$, where the family $\mathcal{Y}' \rar U$ bounds the surfaces $Y'$.
\newline
Let $\mathcal{F}$ denote the reduced divisorial part of the image of $\mathcal{E}$ in $\mathcal{Y}$.
By construction, the components of $\mathcal{E}$ corresponding to $\hat j^*(B_{Y^\qq}+\Delta_{Y^\qq}+D^\qq)$ induce components of $\mathcal{E}$ corresponding to $B_{Y'}+\Delta_{Y'}+D'$.
Then, the claim follows by considering the family $\mathcal{J} \rar (\mathcal{Y},\mathcal{F}) \rar U$.
\end{proof}

\begin{remark} \label{rmk section}
By Step 7 in the proof of Theorem \ref{bdd jacobian}, the rational sections $\hat Y$ are bounded together with the fibrations $J(X) \rar (Y,B_Y+\Delta_Y+D)$.
Thus, for every irreducible component $U_i \subset U$, the corresponding fibration $\mathcal{J}_i \rar \mathcal{Y}_i$ is an elliptic fibration with a rational section.
Furthermore, the rational section restricts to the rational section of $J(X) \rar Y$ for every element of $\mathfrak{P}(n,v,\Phi(\mathcal R),d)$.
\end{remark}

\begin{lemma} \label{lemma supports}
Let $X$ be an element of $\mathfrak{C}(n,v,\Phi(\mathcal R),d)$, and let $(Y,B_Y+\Delta_Y)$ be the associated pair.
Furthermore, let $D$ be as in Proposition \ref{bound cycles} and $(Y',B_{Y'}+\Delta_{Y'})$ be as in Definition \ref{def fibration J}.
Then, the induced morphism $Y' \setminus \Supp(B_{Y'}+D') \rar Y$ is an isomorphism with its image, and its image is included in $Y \setminus \Supp(D)$.
Here, $D'$ denotes the pull-back of $D$ to $D'$.
\end{lemma}

\begin{proof}
By construction, $\Supp(D)$ contains the singular locus of $Y$.
Thus, $\Supp(D')$ contains the preimage of the singular locus of $Y$.
Hence, up to shrinking, we may assume that $Y$ is smooth, and, in particular, $\qq$-factorial.
This guarantees that the exceptional locus is purely divisorial.
\newline
As argued at the beginning of \S~\ref{section elliptic} and in Lemma \ref{lemma cbf jacobian}, there is a divisor $C_{Y'}$ so that
\begin{itemize}
    \item $0 \leq C_{Y'} \leq B_{Y'}$; and
    \item $\K Y'. + C_{Y'} + \Delta_{Y'}$ is relatively ample over $Y$; and
    \item $\K Y'. + C_{Y'} + \Delta_{Y'}-\pi^*(\pi_*(\K Y'. + C_{Y'} + \Delta_{Y'}))$ is $\pi$-exceptional (notice that here we are using the reduction to the case when $Y$ is $\qq$-factorial).
\end{itemize}
By the negativity lemma \cite{KM98}*{Lemma 3.39}, it follows that
\[
\K Y'. + C_{Y'} + \Delta_{Y'} = \pi^*(\pi_*(\K Y'. + C_{Y'} + \Delta_{Y'}))-F,
\]
where $F \geq 0$ is fully supported on the $\pi$-exceptional locus.
Since $B_Y \geq \pi_*(C_{Y'})$, it follows that and $0 \leq C_{Y'} \leq C_{Y'} + F \leq B_{Y'}$.
In particular, $\Supp(B_{Y'})$ contains the $\pi$-exceptional locus.
This concludes the proof.
\end{proof}

\begin{proposition} \label{prop for surfaces}
Let $(X,\Delta)$ be a minimal klt pair with $\dim(X)=2$ and $\kappa(X,\Delta)=1$.
Let $f \colon X \rar Y$ be the Iitaka fibration, and assume it is an elliptic fibration.
Let $B_Y$ be the boundary divisor of the canonical bundle formula for the pair $(X,\Delta)$ and the morphism $f$.
Then, for every closed point $P \in Y \setminus \Supp(B_Y)$, $f$ admits a section \'etale locally at $P$.
\end{proposition}

\begin{proof}
Let $X'$ be the minimal resolution of $X$.
Notice that, if $X' \rar Y$ admits a section \'etale locally, then so does $X \rar Y$ by composing with the morphism $X' \rar X$.
Thus, we may assume that $X$ is smooth.
Since there is an inclusion $f(\Supp(\Delta)) \subset \Supp(B_Y)$, the fibration is a minimal elliptic fibration in the sense of Kodaira over $Y \setminus \Supp(B_Y)$.
Therefore, all the fibers are reduced over $Y \setminus \Supp(B_Y)$.
Then, the claim follows.
\end{proof}

\begin{proof}[Proof of Theorem \ref{main thm}]
We proceed by induction on $n$.
Proposition \ref{prop for surfaces} implies that an analog of Proposition \ref{bound cycles} holds true if $n=2$.
This will guarantee that the base case $n=2$ can be proved following the steps below.
Since the strategy below applies with no changes to the base case $n=2$ and to the inductive step, we do not make mention of the induction in the rest of the proof.
Our strategy follows the one in the proof of \cite{Gro94}*{Theorem 4.3}.
We proceed in several steps.
\newline
{\bf Step 1:} {\it In this step we consider the family of Jacobian fibration guaranteed by Theorem \ref{bdd jacobian} and construct some closed subsets on this family.}
\newline
Let $\mathcal{J} \rar (\mathcal{Y},\mathcal{F}) \rar U$ be the family of fibrations constructed in Theorem \ref{bdd jacobian}.
Up to a stratification of $U$, we may assume that $U$ is smooth.
Similarly, we may assume that the closed subset of $\mathcal{Y}$ where the morphism $\mathcal{J} \rar \mathcal{Y}$ is not smooth does not contain any fiber of $\mathcal{Y} \rar U$.
Let $\mathcal{Z} \subset \mathcal{Y}$ denote a closed subset so that $\mathcal{J} \rar \mathcal{Y}$ is a smooth fibration over $\mathcal{Y} \setminus \mathcal{Z}$.
Up to enlarging $\mathcal{Z}$, by Remark \ref{rmk section}, we may assume that $\mathcal{J} \rar \mathcal{Y}$ has a regular section over $\mathcal{Y} \setminus \mathcal{Z}$.
Up to a further stratification and Noetherian induction, we may assume that $\mathcal{Z}$ does not contain any fiber of $\mathcal{Y} \rar U$.
\newline
{\bf Step 2:} {\it In this step we show that, after shrinking the family according to the closed subsets constructed in Step 1, the fibrations of interest are parametrized by the Tate--Shafarevich group.}
\newline
Define $\mathring{\mathcal{Y}} \coloneqq \mathcal{Y} \setminus (\mathcal{F} \cup \mathcal{Z})$ and $\mathring{\mathcal{J}} \coloneqq \mathcal{J} \times_\mathcal{Y} \mathring{\mathcal{Y}}$.
For an element $J(X) \rar (Y',B_{Y'}+\Delta_{Y'}+D') \in \mathfrak{P}(n,v,\Phi(\mathcal R),d)$ corresponding to the fiber over $u \in U$, we denote by $\mathring{Y} \subset Y'$ the open subset induced by $\mathring{\mathcal{Y}}$.
By Lemma \ref{lemma supports}, we may identify $\mathring{Y}$ with an open subset of $Y$ as well.
Thus, we may write $J(\mathring{X}) \coloneqq J(X) \times_Y \mathring{Y}$ and $\mathring{X}\coloneqq X \times_Y \mathring{Y}$.
Then, by construction and Proposition \ref{bound cycles}, $X \rar Y$ is smooth over $\mathring{Y}$, except possibly at finitely many isolated points.
Call these points $P_1,\ldots,P_k$.
Furthermore, the fibration $J(X) \rar Y'$ is smooth over $\mathring{Y}$ and it admits a section.
By construction, we know that $X \rar Y$ corresponds to an element of $\Sh_{\mathring{Y} \setminus \lbrace P_1,\ldots,P_k \rbrace}(J(X)_{\eta_Y})$.
Then, by Proposition \ref{lemma tecnico DG}, it actually belongs to $\Sh_{\mathring{Y}}(J(X)_{\eta_Y})$.
Notice that, by Lemma \ref{lemma order WC}, this element has order dividing $d!$ in $\Sh_{\mathring{Y}}(J(X)_{\eta_Y})$.
\newline
{\bf Step 3:} {\it In this step we introduce a log resolution of $\mathcal{J} \rar (\mathcal{Y},\mathcal{F} \cup \mathcal{Z}) \rar U$.}
\newline
Up to a further stratification of $U$, we may assume that $\mathcal{Y}$ and $\mathcal{J}$ admit log resolutions $\mathcal{Y}'$ and $\mathcal{J}'$ so that
\begin{itemize}
    \item $\mathcal{J}' \rar \mathcal{J}$ and $\mathcal{Y}' \rar \mathcal{Y}$ are isomorphisms over $\mathring{\mathcal{J}}$ and $\mathring{\mathcal{Y}}$, respectively;
    \item the induced rational map $\mathcal{J}' \drar \mathcal{Y}'$ is an actual morphism which admits a section;
    \item the complement of $\mathring{\mathcal{Y}}$ in $\mathcal{Y}'$ is simple normal crossing over $U$; and
    \item $\mathcal{J}' \rar U$ and $\mathcal{Y'} \rar U$ are smooth.
\end{itemize}
Notice that, by construction, the ramification locus of $\mathcal{J}' \rar \mathcal{Y}'$ is contained in the complement of $\mathring{Y}$.
\newline
{\bf Step 4:} {\it In this step we show that the set of fibrations (without the boundary divisors) we are considering is birationally bounded.
For this purpose, we will refer to several statements in \cite{Gro94}*{\S~4}.
While the results in \cite{Gro94}*{\S~4} are stated for families of elliptic threefolds, we will make use only of the ones that hold without any restriction on the dimension.
One of the needed statements, namely \cite{Gro94}*{Lemma 4.7}, holds in full generality only in dimension 3.
Yet, for this specific statement, we only need the conclusion of the first half of its proof, which holds true with no restriction on the dimension (specifically, we only need the statement when the ramification divisor is removed, i.e., $(\mathcal{X}',\mathcal{S}',\mathcal{T})$ in the notation of \cite{Gro94}*{proof of Lemma 4.7}).}
\newline
By abusing notation, we proceed as if $U$ were irreducible: to be precise, the following steps should be performed on each irreducible component of the family of fibrations.
Now, let $\eta_{\mathcal{Y}'}$ denote the generic point of $\mathcal{Y}'$, and let $\iota \colon \eta_{\mathcal{Y}'} \rar \mathcal{Y}'$ be the natural inclusion.
We regard $\iota_* \mathcal{J}'_{\eta_{\mathcal{Y}'}}$ as a sheaf in the \'etale topology of $\mathcal{Y}'$.
Similarly, its (derived) push-forwards and pull-backs will be understood in the \'etale site.
We write $\pi \colon \mathcal{Y}' \rar U$ and $\mathring{\pi} \colon \mathring{\mathcal{Y}} \rar U$.
Then, by \cite{Gro94}*{Lemma 4.4}, $_{d!} R^1\mathring{\pi}_*(\iota_*\mathcal{J}'_{\eta_{\mathcal{Y}'}})$, the $d!$-torsion subsheaf of $R^1\mathring{\pi}_*(\iota_*\mathcal{J}'_{\eta_{\mathcal{Y}'}})$, is a constructible sheaf on $U$.
By \cite{Mil80}*{Proposition V.1.8}, there exists an open subset $V \subset U$ such that $_{d!} R^1\mathring{\pi}_*(\iota_*\mathcal{J}'_{\eta_{\mathcal{Y}'}})|_V$ is locally constant with finite stalks.
Thus, there is an \'etale cover $V' \rar V$ such that $_{d!} R^1\mathring{\pi}_*(\iota_*\mathcal{J}'_{\eta_{\mathcal{Y}'}})|_{V'}$ is a constant sheaf with stalks isomorphic to a finite group $G$.
For every $v \in V'$ and $g \in G$, there exists an \'etale neighborhood $V_{v,g} \rar V'$ of $v$ such that the element $g$ of the stalk $(_{d!} R^1\mathring{\pi}_*(\iota_*\mathcal{J}'_{\eta_{\mathcal{Y}'}})|_{V'})_{\overline{v}} =G$ is represented by an element of $E_{v,g} \in H^1(V_{v,g} \times_U \mathring{\mathcal{Y}},\iota_*\mathcal{J}'_{\eta_{\mathcal{Y}'}})$.
Fix $g \in G$.
Then, by quasi-compactness, we may find a finite subcover of $\lbrace V_{v,g}|v \in V' \rbrace$ that covers $V'$.
Thus, we produce a collection of conncted \'etale schemes $V_1,\ldots,V_k$ over $V'$ and elements $E_j \in H^1(V_j \times_U \mathring{\mathcal{Y}},\iota_*\mathcal{J}'_{\eta_{\mathcal{Y}'}})$ such that, for all closed points $v \in V$, restricting each $E_j$ to $\mathring{\mathcal{Y}}_v$ gives all elements of $_{d!} \Sh_{\mathring{\mathcal{Y}}_v}((\mathcal{J}'_v)_{\eta_{\mathring{\mathcal{Y}}_v}})$, by \cite{Gro94}*{Lemma 4.7}.
By Noetherian induction, we can repeat the above argument over $U \setminus V$.
Thus, arguing as in the proof of \cite{Gro94}*{Theorem 4.3}, we obtain a family (over a possibly disconnected base) $\tilde{\mathcal{X}} \rar \tilde{\mathcal{Y}} \rar T$ such that, for every $u \in U$ and every element of $_{d!} \Sh_{\mathring{\mathcal{Y}}_u}((\mathcal{J}_u)_{\eta_{\mathring{\mathcal{Y}}_u}})$, there is a closed point $t \in T$ such that the fibration $\tilde{\mathcal{X}}_t \rar \tilde{\mathcal{Y}}_t$ corresponds to such an element.
Furthermore, we have that $\tilde{\mathcal{X}}$ and $\tilde{\mathcal{X}} \rar T$ are smooth.
Thus, we conclude that the varieties underlying the pairs in $\mathfrak{C}(n,v,\Phi(\mathcal R),d)$ are birationally bounded.
\newline
{\bf Step 5:} {\it In this step we retrieve the birational boundedness of the divisors corresponding to the elements of $\mathfrak{C}(n,v,\Phi(\mathcal R),d)$.}
\newline
In Step 4, the variety $\tilde{\mathcal{Y}}$ together with the morphism $\tilde{\mathcal{Y}} \rar T$ were constructed by iterated base changes and stratifications of the original base $U$.
In particular, the simple normal crossing divisor $\mathcal{Y}' \setminus \mathring{Y}$ pulls back to $\tilde{\mathcal{Y}}$.
Call this divisor $\Sigma$, and let $\Xi$ be its pull-back to $\tilde{\mathcal{X}}$.
Now, recall that $\mathcal{Y}' \rar \mathcal{Y}$ is an isomorphism over $\mathring{\mathcal{Y}}$.
Let $\tilde{\mathcal{X}}_t$ be a closed fiber corresponding to the fibration $(X,\Delta) \rar (Y,B_Y+\Delta_Y)$, and let $\K \tilde{\mathcal{X}}_t. + \tilde{\Delta}_t$ denote the crepant pull-back of $\K X. + \Delta$ to $\tilde{\mathcal{X}}_t$.
Notice that, while $\tilde{\mathcal{X}}_t \drar X$ may not be a morphism, the pull-back is crepant, since $\K X. + \Delta \sim \subs \qq,Y. 0$.
Let $\Sigma_t$ denote the restriction of $\Sigma$ to $\tilde{\mathcal{Y}}_t$.
Then, every prime component of $\tilde{\Delta}_t$ or of the exceptional locus of $\tilde{\mathcal{X}}_t \drar X$ that maps to a divisor on $\tilde{\mathcal{Y}}_t$ maps to $\Sigma_t$.
Thus, every other component of $\tilde{\Delta}_t$ or of the exceptional locus of $\tilde{\mathcal{X}}_t \drar X$ is exceptional for $\tilde{\mathcal{X}}_t \rar \tilde{\mathcal{Y}}_t$ in the sense of \cite{Lai11}*{Definition 2.9}.
In order to conclude this step, it suffices to show the boundedness of these divisors.
\newline
Now, define the closed subset $\mathcal{V} \subset \tilde{\mathcal{Y}}$ as the set of points $y$ so that $\dim(\tilde{\mathcal{X}}_y) > 1$.
Let $\mathcal{W}$ be the inverse image of $\mathcal{V}$ in $\tilde{\mathcal{X}}$ with the reduced structure.
Now, we may stratify $T$ so that $\mathcal{W}$ is flat over $T$.
Thus, the divisors that are exceptional for the morphism $\tilde{\mathcal{X}}_t \rar \tilde{\mathcal{Y}}_t$ for some $t \in T$ are bounded, as they deform in the divisorial part of $\mathcal{W}$.
Call $\Omega$ the reduced divisor obtained as the union of $\Xi$ and the divisorial part of $\mathcal{W}$.
This divisor bounds the strict transform of the support of $\Delta$ and the exceptional divisor of $\tilde{\mathcal{X}}_t \drar X$ as required.
\newline
{\bf Step 6:} {\it In this step we perform some suitable modifications to the birationally bounding family in order to prepare it for the run of a suitable MMP.}
\newline
In Step 5, we achieved the birational boundedness of the fibrations of interest.
In order to retrieve (up to codimension 1 indeterminacy) the fibrations $(X,\Delta) \rar Y$ in $\mathfrak{C}(n,v,\Phi(R),d)$, we will use that the pairs $(X,\Delta)$ are good minimal models.
Thus, the ultimate goal of the proof is to replace the family $(\tilde{\mathcal{X}},\Omega) \rar T$ constructed in Step 5 with a new family where each fiber is a good minimal model of a fiber of $(\tilde{\mathcal{X}},\Omega) \rar T$.
This final replacement will be performed in Step 7.
In this step, we will improve the family constructed in Step 5 to have the additional properties needed to run the MMP in families.
In particular, two key properties are needed.
First, we need that, if $(\tilde{\mathcal{X}}_t,\Omega_t)$ corresponds to $(X,\Delta)$, the rational map $\tilde{\mathcal{X}}_t \drar X$ is a rational contraction.
To achieve this, we will use the fact that $(X,\Delta)$ is $\epsilon$-log canonical for a uniform $\epsilon>0$, and thus this property is guaranteed up to replacing $(\tilde{\mathcal{X}}_t,\Omega_t)$ with the terminalization of $(\tilde{\mathcal{X}}_t,\left(1-\frac{\epsilon}{2} \right)\Omega_t)$.
Secondly, in order to apply \cite{HMX18b}*{Theorem 1.9.1} in Step 7, we need for this replacement to occur on the whole family, and for the new birational model of $(\tilde{\mathcal{X}},\Omega) \rar T$ to be log smooth over $T$.
To this end, we will consider log resolutions and terminalizations of the total space $\tilde{\mathcal{X}}$, and we will recur to stratifications of $T$ and Noetherian induction to guarantee that resolutions and terminalizations of the total space restrict to resolutions and terminalizations fiberwise.

Up to a further resolution, we may assume that $\Omega$ is simple normal crossing.
In doing so, we include in $\Omega$ the divisors extracted in this process.
By Corollary \ref{cor bound sings}, each $(X,\Delta)$ is $\epsilon$-log canonical for a fixed $\epsilon > 0$.
Thus, we may replace $(\tilde{\mathcal{X}},\Omega)$ with a log resolution of the terminalization of $(\tilde{\mathcal{X}},(1-\frac{\epsilon}{2})\Omega)$.
Furthermore, up to a stratification of $T$, we may assume that this process induces a terminalization fiberwise.
By abusing notation, we call $(X,\Omega)$ the outcome of this process.
This process guarantees that no divisor is contracted by the birational map $X \drar \tilde{\mathcal{X}}_t$.
Then, up to a stratification, we may assume that every stratum of $\Omega$ is smooth over $T$.
Finally, by \cite{Kol13}*{Claim 4.38.1}, we may perform an \'etale base change of a stratification of $T$ so that every prime component of $\Omega$ restrict to a prime divisor fiberwise.
This reduction is needed in order to be able to assign coefficients to the divisors $\Omega_t$ with a choice of coefficients for $\Omega$.
\newline
{\bf Step 7:} {\it In this step we conclude the proof.}
\newline
Fix $(X,\Delta) \rar Y$ in $\mathfrak{C}(n,v,\Phi(\mathcal R),d)$, and let $\tilde{\mathcal{X}}_t$ be the corresponding closed fiber of $\tilde{\mathcal{X}} \rar T$.
Then, the strict transform of $\Delta$ is supported on $\Omega_t$.
Let $\Gamma_t$ be the divisor supported on $\Omega_t$ obtained as follows: it is the sum of the strict transform of $\Delta$ and the divisors that are exceptional for $\tilde{\mathcal{X}}_t \drar X$, the latter having coefficients $1-\frac{\epsilon}{2}$.
Then, by construction, $(X,\Delta)$ is a good minimal model for $(\tilde{\mathcal{X}}_t,\Gamma_t)$.
Now, let $(\tilde{\mathcal{X}},\Gamma)$ be the pair supported on $\tilde{\mathcal{X}}$ whose restriction to $\tilde{\mathcal{X}}_t$ is $(\tilde{\mathcal{X}}_t,\Gamma_t)$ (notice that this process may actually involve just one of the finitely many connected components of $\tilde{\mathcal{X}}$).
Then, by \cite{HMX18b}*{Theorem 1.9.1}, $(\tilde{\mathcal{X}},\Gamma)$ admits a relative good minimal model over $T$ that induces a good minimal model fiberwise.
Furthermore, the relative ample model induces the ample model fiberwise.
Call $(\hat{\mathcal{X}},\hat{\Gamma})$ the minimal model, and let $\hat{\mathcal{Y}}$ be the relative ample model.
Then, by \cite{HX13}*{Lemma 2.4},
$(\hat{\mathcal{X}}_t,\hat{\Gamma}_t)$ is isomorphic in codimension one to $(X,\Delta)$, and $\hat{\mathcal{Y}}_t$ is isomorphic to $Y$.
\newline
Now, notice that $\tilde{\mathcal{X}} \rar T$ is a morphism of schemes of finite type.
In particular, they have finitely many irreducible components.
Furthermore, $\Omega$ has finitely many irreducible components.
Finally, the possible coefficients involved in the construction of $\Gamma_t$ as above are finite.
Indeed, $\Phi(\mathcal{R}) \cap [0,1-\epsilon]$ is a finite set.
Thus, up to performing all the possible combinations of choices involving irreducible components of $\tilde{\mathcal{X}}$ and all the possible coefficients to assign to $\Omega$, the above process retrieves every $(X,\Delta)$ up to isomorphism in codimension one.
This implies that there exist schemes of finite type $\overline{\mathcal{X}}$, $\overline{\mathcal{Y}}$ and $S$, a reduced divisor $\overline{\mathcal{B}}$ on $\overline{\mathcal{X}}$, and projective morphisms $(\overline{\mathcal{X}},\overline{\mathcal{B}}) \rar \overline{\mathcal{Y}} \rar S$ that bound in codimension one the pairs $(X,\Delta)$ and bound the varieties $Y$.
In order to conclude, we need to show that this family also bounds the divisors arising from the canonical bundle formula applied to any $(X,\Delta) \rar Y$.
Up to a stratification of $S$, this can be achieved by Corollary \ref{stratify cbf}.
This concludes the proof.
\end{proof}

\begin{proof}[Proof of Theorem \ref{thm intro}]
It is immediate from Theorem \ref{theorem fano case} and Theorem \ref{main thm}.
\end{proof}

\end{document}